\documentclass[12pt]{article}
\usepackage{amsmath}
\usepackage{amssymb}
\usepackage{amsthm}
\usepackage{bbm}
\usepackage[dvipdfmx, usenames]{color}
\usepackage{mathrsfs}
\usepackage{bm}
\usepackage[numbers]{natbib}  
\usepackage{tabularx}
\theoremstyle{plain} 
\newtheorem{thm}{Theorem}[section]
\newtheorem{lem}[thm]{Lemma}
\newtheorem{prop}[thm]{Proposition}
\newtheorem{cor}[thm]{Corollary}
\theoremstyle{plain}
\newtheorem{defn}[thm]{Definition}
\newtheorem{exam}[thm]{Example}
\newtheorem{rem}[thm]{Remark}

\numberwithin{equation}{section}
\DeclareMathOperator*{\essinf}{ess\,inf}
\DeclareMathOperator*{\esssup}{ess\,sup}
\usepackage{hyperref}

\title{
The Mimura Integral: A Unified Framework for Riemann and Lebesgue Integration
}

\author{Yoshifumi MIMURA\thanks{mmryshfm@gmail.com}
\\ 
Independent Researcher, Tokyo, Japan\\
}

\date{}

\begin{document}

\maketitle

\begin{abstract}
An integral on Euclidean space, equivalent to the Lebesgue integral, 
is constructed by extending the notion of Riemann sums. 
In contrast to the Henstock--Kurzweil and McShane integrals, 
the construction recovers the full measure-theoretic structure 
--- outer measure, inner measure, and measurable sets --- 
rather than merely reproducing integration with respect to the Lebesgue measure. 
Whereas the classical approach to Lebesgue theory proceeds through 
a two-layer framework of measure and integration, 
these layers are unified here into a single framework, 
thereby avoiding duplication. 
Compared with the Daniell integral, the method is more concrete and accessible, 
serving both as an alternative to the Riemann integral 
and as a natural bridge to abstract Lebesgue theory.
\end{abstract}

\noindent
{\sc keywords}:{\ Lebesgue integral; Riemann integral; Measure theory; Integration theory; Real analysis; Mathematical education}


\newtheorem*{thm*}{Theorem}
\newtheorem*{lem*}{Lemma}
\newtheorem*{prop*}{Proposition}
\newtheorem*{cor*}{Corollary}
\theoremstyle{plain}
\newtheorem*{defn*}{Definition}
\newtheorem*{exam*}{Example}
\newtheorem*{rem*}{Remark}

\section*{Introduction}

The Lebesgue integral is indispensable in modern analysis. Nevertheless, the standard construction that begins with measure theory can seem circuitous to students eager to learn integration. For this reason, even after the establishment of Lebesgue integration, refinements of Riemann's definition --- such as the Henstock--Kurzweil and McShane integrals --- have been proposed. On the other hand, there are also approaches like the Daniell integral that do not presuppose a measure and instead take the integral itself (a positive linear functional) as the point of departure. In this paper we construct an integral by a method different from all of these and show that it is equivalent to the Lebesgue integral. In particular, we recover the measure by integrating characteristic functions of sets, thereby introducing measure-theoretic notions -- outer measure, inner measure, measurable sets -- in a natural way. In other words, we furnish a unified framework for measure and integration.

We first recall the key ideas and distinctions among the Riemann, Henstock--Kurzweil, and McShane integrals. 
For a function 
$f:[a,b] \to \mathbb{R}$, 
consider a family of closed intervals $\{[a_m,b_m]\}_{m=1}^N$ 
and a family of points $\{x_m\}_{m=1}^N$, 
and the Riemann sum 
\[
S:=
\sum\limits_{m=1}^N f(x_m)(b_{m}-a_{m}). 
\]
Assume 
\noindent
\begin{itemize}

\item[\rm{(i)}] $[a,b]=\bigcup\limits_{m=1}^N [a_m, b_m]$, $N \in \mathbb{N}$. 

\item[\rm{(ii)}] $(a_m, b_m) \cap (a_n, b_n)=\varnothing$ if $m \neq n$. 

\item[\rm{(iii)}] $x_m \in [a,b]$. 

\end{itemize}
Under these standing assumptions, the following definitions are standard.

\noindent
\begin{itemize}

\item[\rm{(I)}] Riemann integral \cite[Definition 4.11]{s-k}\,: \\[1ex]
There exists $A$ such that for every $\varepsilon>0$ 
there is $\delta>0$ with  
\[ x_m \in [a_m, b_m] \subset (x_m-\delta, x_m+\delta)\] 
for all pairs 
$\{(x_m, [a_m, b_m])\}$,  implying 
\[
|S-A|<\varepsilon. 
\]
We then call
$A$ the Riemann integral of $f$ over $[a,b]$.

\item[\rm{(II)}] 
Henstock--Kurzweil integral \cite[Definition 4.12]{s-k}\,: \\[1ex]
There exists $B$ such that for every $\varepsilon>0$ 
there is a function (called gauge) 
$\delta:[a,b] \to (0, \infty)$ with
\[ x_m \in [a_m, b_m],\quad [a_m, b_m] \subset (x_m-\delta(x_m), x_m+\delta(x_m)) \]
for all $\{(x_m, [a_m, b_m])\}$,  implying 
\[
|S-B|<\varepsilon. 
\]
We then call $B$ the Henstock--Kurzweil integral of $f$ over $[a,b]$

\item[\rm{(III)}] McShane integral \cite[Definition 5.4]{s-k}\,: \\[1ex]
There exists $C$ such that for every $\varepsilon>0$
there is a function (called gauge) 
$\delta:[a,b] \to (0, \infty)$ with  
\[ [a_m, b_m] \subset (x_m-\delta(x_m), x_m+\delta(x_m)), 
\quad (\text{we do not require $x_m \in (a_m,b_m)$})  \]
for all $\{(x_m, [a_m, b_m])\}$,  implying  
\[
|S-C|<\varepsilon. 
\]
We then call $C$ the McShane integral of $f$ over $[a,b]$ 

\end{itemize}

\begin{rem*} 
Under {\rm (II)}, the admissible families $\{(x_m, [a_m, b_m])\}$ 
form a subset of those allowed in {\rm(III)}. 
Hence McShane integrability implies Henstock--Kurzweil integrability 
{\rm (}the converse fails{\rm)}.
\end{rem*}

Thus the conditions imposed on the points (called tags) 
$\{x_m\}$ and on the partition 
determine the class of integrable functions. 
Riemann integrability is the narrowest; McShane integrability strictly extends Riemann integrability and coincides with Lebesgue integrability; Henstock--Kurzweil integrability strictly extends Lebesgue integrability (it does not require absolute integrability). For example,
\[
f(x):=
\begin{cases}
\displaystyle \frac{1}{x}\sin{\frac{1}{x^2}}, & x \in (0,1], \\[1.5ex]
0, & x=0,
\end{cases}
\]
is Henstock--Kurzweil integrable 
but neither absolutely integrable nor Lebesgue integrable. In symbols,
\[ \text{Riemann}\subset \text{Lebesgue} 
=\text{McShane} \subset \text{Henstock--Kurzweil}\]

However, each of these three integrals (R, HK, MS) lacks an explicit measure-theoretic viewpoint. The Henstock--Kurzweil integral is not well suited 
to constructing normed spaces such as Lebesgue spaces $L^p(\mathbb{R}^d)$, 
and McShane's theory faces limitations without Lebesgue measure theory. 
Motivated by the connection to measure, we first define Lebesgue null sets independently, prior to outer measure, and then 
--- 
for a nonnegative function $f\geq 0$ 
---
we introduce the following 
\lq \lq extension of partitions\rq \rq and 
\lq \lq constraint on tags\rq \rq:

\begin{itemize}

\item[\rm{(M1)}] Measure-theoretic decomposition\,:\\[1ex]
Decompose an open interval $(a,b)$ 
into a countable disjoint union 
$(a,b)=\bigsqcup\limits_{m=1}^{\infty}(a_m, b_m)$ 
in the sense of \lq \lq almost everywhere\rq \rq\,; 
namely, 
the symmetric difference 
of $(a,b)$ and $\bigsqcup\limits_{m=1}^{\infty}(a_m, b_m)$ 
is 
a Lebesgue null set.

\item[\rm{(M2)}] Essential-bounds tagging\,: \\[1ex]
In the extended Riemann sum 
$\sum\limits_{m=1}^{\infty} f(x_m)(b_m-a_m)$, choose tags 
$x_m \in (a_m, b_m)$ satisfying 
\begin{equation*}
\alpha_m \leq f(x_m) \leq \beta_m, 
\end{equation*}
where 
$\alpha_m:=\essinf\limits_{x \in (a_m, b_m)}f(x)$ and  
$\beta_m:=\esssup\limits_{x \in (a_m, b_m)}f(x)$, see \S 1 for definition.
\end{itemize}
In line with the spirit of these conventions, we do \emph{not} require the containment 
\[ (a_m,b_m)\subset(x_m-\delta,x_m+\delta)\] for every admissible tagged decomposition. 
Instead, for any measure-theoretic decomposition $P=\{(a_m,b_m)\}$ as in {\rm(M1)}, set
\[
L(f;P):=\sum_{m=1}^\infty \alpha_m(b_m-a_m), 
\qquad 
U(f;P):=\sum_{m=1}^\infty \beta_m(b_m-a_m),
\]
and define the lower and upper integrals by
\[
\underline{\int_a^b} f:=\sup_{P} L(f;P), 
\qquad 
\overline{\int_a^b} f:=\inf_{P} U(f;P).
\]
We say that $f$ is \emph{Mimura integrable} on $(a,b)$ if 
the upper integral coincides with the lower integral, 
in which case their common value is called 
the \emph{Mimura integral} of $f$ over $(a,b)$.
\noindent
(As usual, signed functions are handled via the decomposition $f=f^{+}-f^{-}$ whenever the two parts are Mimura integrable.)

The Mimura integral 
is formulated on Euclidean spaces 
$\mathbb{R}^d\ (d \geq 1)$ 
and 
shown to be equivalent to the Lebesgue integral. 
In particular, 
by defining the integral of a characteristic function 
$\chi_E$ of a set $E \subset \mathbb{R}^d$ to be $\mu(E)$, 
one recovers the measure. More precisely, 
when $f$ is a characteristic function, 
the upper (resp. lower) integral corresponds to 
the outer (resp. inner) measure, 
and the sets $E$ 
for which outer and inner measures coincide 
are precisely those satisfying Carath\'eodory's condition 
\[
\overline{\mu}(A)= \overline{\mu}(A\cap E) + \overline{\mu}(A \cap E^c), \quad 
\forall A \subset \mathbb{R}^d, 
\]
where $\overline{\mu}(A)$, the upper integral of $\chi_A$ 
coincides with the Lebesgue outer measure of $A$. 
Thus, having first given a succinct, 
Riemann-style definition of the integral, 
one can subsequently develop measure theory. 
This in turn allows the treatment of 
the Monotone Convergence Theorem, Fatou's Lemma, 
and the Dominated Convergence Theorem.

While the Daniell integral (see \cite{d}, \cite{stone}, \cite{b}) likewise yields a theory equivalent to Lebesgue integration with integration preceding measure, it is often perceived as abstract by beginners and is best approached after one is already familiar with the Lebesgue integral via standard measure theory. By contrast, the Mimura integral rests on a concrete construction that retains a Riemannian flavor and is thus well suited to novices.

As for the roles of the Riemann and Lebesgue integrals, in applications the Riemann integral can arguably be replaced by the Lebesgue integral. Yet the path from measure theory to integration can be both long and opaque, so pedagogically the Riemann integral remains entrenched. Typically one develops substantial results in calculus using Riemann integration and later switches to Lebesgue integration, often 
without re-proving overlapping theorems except where extensions are needed. 
In this respect, the Mimura integral allows one to carry over Riemann's methods unchanged, grounded in the understanding that Lebesgue null sets do not affect the integral. The Mimura integral not only bridges Riemann and Lebesgue but may also serve as a viable substitute for the Riemann integral.

Finally, compared directly with the Lebesgue integral, 
the Mimura integral is equivalent to it; 
at the level of construction, the choice between them is largely a matter of taste. 
The present development is specialized to integration on Euclidean spaces; 
once abstract settings are included, it covers only a small part of Lebesgue theory.
Even so, Lebesgue theory has a genuinely two-layer structure 
--- measure and integration --- 
that standard expositions tend to blur. 
On the measure side, the outer measure alone drives the story: 
one can develop the measure-theoretic layer purely from outer measure, 
without appealing to inner measure. On the integration side, 
however, the integral is defined via simple functions 
and monotone extension from below
--- an intrinsically inner-measure-type procedure. 
Because this is presented as a definition of the integral 
rather than as a measure construction, 
the \lq \lq inner-measure\rq \rq content is easy to miss. 
In the Mimura framework, 
upper and lower integrals make this latent split explicit: the upper integral corresponds to the outer-measure layer, the lower integral to the inner-measure layer, and the two operate symmetrically and transparently. Besides sharpening one's understanding of Lebesgue integration, this symmetry, we expect, substantially reduces the educational cost.

\clearpage
\tableofcontents          
\clearpage

\section{Integrals for Nonnegative Functions}

As in the Lebesgue theory, we begin by defining the integral for nonnegative functions. 
Generalizing the finite partitions of $d$-dimensional intervals used in the Riemann integral, 
we start by introducing a countable partition of an open set $\mathcal{O}$. 
The partition need not be exact; 
it is more efficient to require only an almost-everywhere partition 
(i.e., equality except on a Lebesgue null set). 
We first specify the $d$-dimensional intervals.

\begin{defn}[Open $d$-dimensional intervals, their measure, and size]

\noindent
\begin{itemize}

\item[\rm{(i)}] Open $d$-dimensional interval\,{\rm :} 
$
I=\prod\limits_{n=1}^d(a_n, b_n)
\subset \mathbb{R}^d$. 

\item[\rm{(ii)}] Measure of an open $d$-dimensional interval\,{\rm :} 
$
|I|:=\prod\limits_{n=1}^{d}(b_n-a_n) \geq 0
$. 

\item[\rm{(iii)}] Size of an open $d$-dimensional interval\,{\rm :} 
$
\mathrm{size}(I):= \sum\limits_{n=1}^d(b_n-a_n) \geq 0
$. 

\item[\rm{(iv)}] 
Let $\mathscr{I}(\mathbb{R}^d)$ denote the collection of all open 
$d$-dimensional intervals. 
\\[1ex]
We include the empty set
$\varnothing \in \mathscr{I}(\mathbb{R}^d)$ 
and stipulate 
$|\varnothing|=0$ and $\mathrm{size}(\varnothing)=0$.

\end{itemize}

\end{defn}

Lebesgue null sets can be defined without appealing to outer measure. 
\begin{defn}[Lebesgue null set]
A set $E \subset \mathbb{R}^d$ is a Lebesgue null set if, 
for every $\varepsilon>0$ there exists a sequence of open 
$d$-dimensional intervals  
$\{I_m\}_{m=1}^{\infty} \subset \mathscr{I}(\mathbb{R}^d)$ such that 
\[ E \subset \bigcup\limits_{m=1}^{\infty}I_m \quad \text{and } \quad 
\sum\limits_{m=1}^{\infty}|I_m|<\varepsilon. \]
We denote by $\mathcal{N}(\mathbb{R}^d)$ 
the collection of all Lebesgue null sets in $\mathbb{R}^d$. 
\end{defn}

\begin{exam}
The set of rational numbers $\mathbb{Q}\subset \mathbb{R}$
is a Lebesgue null set.
\end{exam}

For working with Lebesgue null sets, the following notation will be convenient.
\begin{defn}[Almost everywhere] 
A property $P(x)$ is said to hold 
almost everywhere 
on a set $E$ 
if the set of points $x \in E$ where 
$P(x)$ fails is a Lebesgue null set. We write 
\[
P(x)\ \ \text{a.e. }\quad \text{or}\quad P(x)\ \ \text{a.e. on }E. 
\]
\end{defn}

The following result appears in many texts on measure and integration; 
see, for example, \cite[Lemma 6.9]{a-dp-m}, 
\cite[Lemma 1.4.2]{c}, \cite[Lemma 2.43]{Folland}, 
\cite[Lemma 3.44]{s-k}, \cite[Theorem 1.4]{s-s}. 
In this paper we state it in the form below, 
taking into account \lq \lq a.e.\rq \rq statements and the size of intervals.

\begin{lem}[Open set decomposition] 
For any open set $\mathcal{O} \subset \mathbb{R}^d$, 
there exists a sequence 
$\{I_m\}_{m=1}^{\infty} \subset \mathscr{I}(\mathbb{R}^d)$
such that 
\[
\mathcal{O} = \bigsqcup_{m=1}^{\infty}I_m\ \ \text{a.e.}
\qquad 
\left 
(
\mathcal{O}=\bigcup_{m=1}^{\infty}I_m\ \ \text{a.e.},\  I_m \cap I_n =\varnothing \text{ for }m\neq n
\right )
\] 
In particular, 
$\sup\limits_{m \in \mathbb{N}}\mathrm{size}(I_m)$ 
can be made arbitrarily small. 
Conversely, for any sequence 
$\{I_m\}_{m=1}^{\infty} \subset \mathscr{I}(\mathbb{R}^d)$, 
the disjoint union 
\[
\mathcal{O}' = \bigsqcup_{m=1}^{\infty}I_m
\]
is an open set.
\end{lem}

\begin{proof} 
For $m=(m_1,\dots,m_d)\in\mathbb{Z}^d$ and 
$n\in\mathbb{N}$, 
define the half-open $d$-dimensional intervals 
\[
I_{m,n}:=\prod_{k=1}^d\,[m_k 2^{-n},\, (m_k+1)2^{-n}). 
\]
For each fixed $n$ 
\[
\mathbb{R}^d=\bigsqcup_{m\in\mathbb{Z}^d} I_{m,n}
\]
Given $\varepsilon>0$, choose $N\in\mathbb{N}$ with $d\,2^{-N}<\varepsilon$. 
Then for all $n\ge N$, 
\[ \mathrm{size}(I_{m,n})=d\,2^{-n}\le d\,2^{-N}<\varepsilon. \]
Set 
\[
\mathcal{J}_N:=\{\,I_{m,N} : m \in \mathbb{Z}^d,  I_{m,N}
\subset \mathcal{O}\,\}
\]
and for $\ell\geq 1$ define inductively 
\[
\mathcal{J}_{N+\ell}
:=\Bigl\{\,I_{m,N+\ell} : m\in\mathbb{Z}^d, I_{m,N+\ell}\subset 
\mathcal{O}\cap
 \Bigl (\bigcup_{k=N}^{N+\ell-1}\ \bigcup_{J\in\mathcal{J}_k} J \Bigr )^c\Bigr\}. 
\]
Let $\mathcal{J}:=\bigcup\limits_{\ell=0}^\infty \mathcal{J}_{N+\ell}$. 
Then $\mathcal{J}$
is a pairwise disjoint, 
at most countable family of half-open $d$-dimensional intervals, 
and since the boundary of each such interval 
is a Lebesgue null set, we have
\[
\mathcal{O}=\bigsqcup_{J \in \mathcal{J}}J, 
\quad 
\bigsqcup_{J \in \mathcal{J}}J
=\bigsqcup_{m=1}^{\infty}I_m\ \ \text{a.e.}
\]
for a sequence of open $d$-dimensional intervals 
$\{I_m\}\subset \mathscr{I}(\mathbb{R}^d)$. 
This yields the stated decomposition, with 
$\mathrm{size}(I_{m,n})<\varepsilon$. 
The converse statement is immediate, 
since a (countable) disjoint union of open $d$-dimensional intervals is open. 
\end{proof}

To encompass all possible partitions, 
we define the family of countable partitions of an open set $\mathcal{O}$. 
\begin{defn}[Partitions of an open set by open $d$-dimensional intervals]
For an open set $\mathcal{O}$, 
we denote 
by 
$\Pi(\mathcal{O})$ 
the collections of all countable families 
$\{I_m\}_{m=1}^{\infty} \subset \mathscr{I}(\mathbb{R}^d)$
such that 
\[
\mathcal{O}=\bigsqcup_{m=1}^{\infty}I_m\ \ \text{a.e.}
\]

\end{defn}

Next we introduce the essential supremum and infimum. 
For convenience we fix the following notation for super-level and sub-level sets.
\begin{defn}[Super-level and sub-level sets]
\noindent 
\begin{itemize}
\item[{\rm(i)}] super-level set\,{\rm:} \quad
$\{f>a\}:=\{x\in \mathbb{R}^d : f(x)>a\}$. 
\item[{\rm(ii)}] sub-level set\,{\rm:}\quad
$\{f<a\}:=\{x\in \mathbb{R}^d : f(x)<a\}$. 
\end{itemize}
\end{defn}

\begin{defn}[Essential supremum and infimum]
\label{ess_inf+sup}
Given $A \subset \mathbb{R}^d$, 
define 
\noindent
\begin{itemize}

\item[{\rm(i)}] essential supremum{\rm:}\quad 
$\esssup\limits_{x\in A}f(x)
:=\, \inf \bigl \{ M \in \mathbb{R}:  \{f>M\}\cap A \in \mathcal{N}(\mathbb{R}^d) \bigr \}$. 

\item[{\rm(ii)}] essential infimum{\rm:}\quad 
$\essinf\limits_{x\in A}f(x):=\, \sup\{ m\in \mathbb{R} :  \{f<m\}\cap A \in \mathcal{N}(\mathbb{R}^d) \}$. 

\end{itemize}
Equivalently, 
\[
\essinf_{x \in A}f(x) \leq 
f(x) 
\leq 
\esssup_{x \in A}f(x) \ \ \text{a.e. on }A. 
\]
\end{defn}

Together with countable partitions of open sets, 
we define the lower and upper integrals as follows.
\begin{defn}[Lower and upper integrals for nonnegative functions]
Let $\mathcal{O}\subset \mathbb{R}^d$ be an open set and  
$f:\mathcal{O} \to \mathbb{R}$ with $f \geq 0$ a.e. 
For 
$\{I_m\}\in \Pi(\mathcal{O})$, set 
 \[ \alpha_m:=\essinf\limits_{x \in I_m}f(x), \quad 
\beta_m:=\esssup\limits_{x \in I_m}f(x). \]
Define 
\begin{itemize}
\item[\rm{(i)}] lower integral\,{\rm:}\quad 
$\displaystyle 
\underline{\int_{\mathcal{O}}}f(x)\,d_{\mathrm{M}}x:=
\sup\limits_{\{I_m\} \in \Pi(\mathcal{O})}
\sum\limits_{m=1}^{\infty}\alpha_m|I_m|
$. 

\item[\rm{(ii)}] upper integral\,{\rm:}\quad 
$\displaystyle 
\overline{\int_{\mathcal{O}}}f(x)\,d_{\mathrm{M}}x
:=
\inf\limits_{\{I_m\} \in \Pi(\mathcal{O})}
\sum\limits_{m=1}^{\infty}\beta_m|I_m|
$. 
\end{itemize}
For brevity we also write 
\[
L(f, \{I_m\})
:=\sum\limits_{m=1}^{\infty}\alpha_m|I_m|, \quad 
U(f, \{I_m\})
:=\sum\limits_{m=1}^{\infty}\beta_m|I_m|. 
\]

\end{defn}

We define the integral by requiring the lower and upper integrals to coincide.
\begin{defn}[Integrable functions]
Let $\mathcal{O} \subset \mathbb{R}^d$ be open and 
$f:\mathcal{O} \to \mathbb{R}$ with $f\geq 0$ a.e. 
If 
\[ A:=\overline{\int_{\mathcal{O}}}f(x)\,d_{\mathrm{M}}x=
\underline{\int_{\mathcal{O}}}f(x)\,d_{\mathrm{M}}x<\infty, \]
then $f$ is said to be Mimura integrable on $\mathcal{O}$, 
written 
$f \in \mathcal{M}^1(\mathcal{O})$. 
The common value $A$ is called the
Mimura integral of $f$ over $\mathcal{O}$ and is denoted 
\[ \int_{\mathcal{O}}f(x)\,d_{\mathrm{M}}x.\]
\end{defn}

\begin{rem}
As with other integrals, 
the choice of the integration variable is immaterial, 
so we freely use
\[
\int_{\mathcal{O}}f(x)\,d_{\mathrm{M}}x, \quad 
\int_{\mathcal{O}}f(y)\,d_{\mathrm{M}}y, \quad 
\int_{\mathcal{O}}f\,d_{\mathrm{M}}. 
\]
\end{rem}

\begin{rem}
It will be shown later that
$\mathcal{M}^1(\mathcal{O})$
coincides with the usual class of Lebesgue integrable functions 
$\mathcal{L}^1(\mathcal{O})$. 
\end{rem}

\begin{exam}[Integrability of the Dirichlet function]
The Dirichlet function is not Riemann integrable, 
but under our definition it is readily seen to be integrable with value 
\[
f(x):=\chi_{\mathbb{Q}}(x), \quad 
\int_{\mathbb{R}}f(x)\,d_{\mathrm{M}}x=0. 
\]
\end{exam}

\section{Finite-Measure Sets $\mathcal{M}(\mathbb{R}^d)$}

Whereas the Lebesgue theory proceeds from a construction of measure to a construction of the integral, we have already defined an integral without first constructing a measure. Unlike the standard approach, we can now \emph{define} a measure by means of the integral. 

The following definition identifies sets with functions via their characteristic functions. 
\begin{defn}[Characteristic functions]
For a set $E$, define its characteristic function by
\[ 
\chi_{E}(x):=
\begin{cases}
1, & (x \in E), \\[1ex]
0, & (x \notin E). 
\end{cases}
\]
\end{defn}

\begin{defn}[Outer and inner measures]
For $E\subset \mathbb{R}^d$, 
define the outer and inner measures by
\begin{itemize}
\item[\rm{(i)}] outer measure\,{\rm:}\quad 
$\displaystyle 
\overline{\mu}(E):=\overline{\int_{\mathbb{R}^d}}\chi_{E}(x)\,d_{\mathrm{M}}x$. 
\item[\rm{(ii)}] inner measure\,{\rm:}\quad 
$\displaystyle 
\underline{\mu}(E):=\underline{\int_{\mathbb{R}^d}}\chi_{E}(x)\,d_{\mathrm{M}}x
$. 
\end{itemize}
\end{defn}

\begin{defn}[Finite-measure sets and their measures]
We say $E\subset \mathbb{R}^d$ 
has \emph{finite measure} if 
\[ \overline{\mu}(E)=\underline{\mu}(E)<\infty\]
and write $E \in \mathcal{M}(\mathbb{R}^d)$. 
The common value 
$\mu(E):=\overline{\mu}(E)=\underline{\mu}(E)$
is the \emph{measure} of $E$.  
\end{defn}

\begin{rem}
$E \in \mathcal{M}(\mathbb{R}^d) \iff \chi_E \in \mathcal{M}^1(\mathbb{R}^d)$. 
The class of finite-measure sets will be shown to coincide with the usual Lebesgue measurable sets of finite measure {\rm(}it is not all Lebesgue measurable sets{\rm)}. 
\end{rem}

\begin{defn}[$\mathcal{O}(\mathbb{R}^d)$ and $\mathcal{F}(\mathbb{R}^d)$]
Let $\mathcal{O}(\mathbb{R}^d)$ denote 
the family of open subsets of $\mathbb{R}^d$, 
and let $\mathcal{F}(\mathbb{R}^d)$ denote 
the family of closed subsets of $\mathbb{R}^d$. 
\end{defn}

From the definitions of finite-measure sets 
and of the measure, measurability and measure of open intervals 
and open sets follow immediately. 
In the standard Lebesgue theory, 
one first introduces Carath\'eodory measurability and then 
verifies that intervals --- the starting point for outer measure --- 
are measurable; this is somewhat tedious. 
Under our definitions, it is straightforward to check that intervals lie in 
$\mathcal{M}(\mathbb{R}^d)$, and 
we will later show 
that elements of 
$\mathcal{M}(\mathbb{R}^d)$ 
satisfy Carath\'eodory's condition.

\begin{prop}
If $\mathcal{O}\subset\mathbb{R}^d$ 
is a bounded open set, then
$\mathcal{O}\in\mathcal{M}(\mathbb{R}^d)$. 
In particular, for any $I\in\mathscr{I}(\mathbb{R}^d)$
we have $\mu(I)=|I|$; moreover, if 
$\mathcal{O}=\bigsqcup\limits_{m=1}^{\infty} I_m\ \text{a.e.}$, then 
\[
\mu(\mathcal{O})=\sum_{m=1}^{\infty}|I_m| .
\]
\end{prop}

\begin{proof}
Write 
$\mathcal{O}= \bigsqcup\limits_{m=1}^{\infty}I_m$ a.e.  
By the definitions of lower and upper integrals,
\begin{equation*} 
\begin{split} 
\sum_{m=1}^{\infty}|I_m|=L(\chi_{\mathcal{O}}, \{I_m\})
\leq \underline{\mu}(\mathcal{O}) 
\leq \overline{\mu}(\mathcal{O}) 
\leq U(\chi_{\mathcal{O}}, \{I_m\})= \sum_{m=1}^{\infty}|I_m|. 
\end{split}
\end{equation*}
Hence for every countable partition $\{I_m\}$ of $\mathcal{O}$, 
\[\overline{\mu}(\mathcal{O}) 
= \underline{\mu}(\mathcal{O}) 
=\sum_{m=1}^{\infty}|I_m|, 
\]
so $\mathcal{O} \in \mathcal{M}(\mathbb{R}^d)$ and 
$
\mu(\mathcal{O}) =\sum\limits_{m=1}^{\infty}|I_m|
$. 
\end{proof}

\begin{cor}
Given any sequence of open sets $\{\mathcal{O}_n\}$
and any $N\in\mathbb{N}$, if 
$
\bigcap\limits_{n=1}^{N}\mathcal{O}_n$ 
and $\bigcup\limits_{n=1}^{\infty}\mathcal{O}_n$
are bounded, 
then they belong to $\mathcal{M}(\mathbb{R}^d)$. 
\end{cor}

\begin{proof}
Finite intersections and countable unions of open sets are open.
\end{proof}

\begin{prop}[Monotonicity of the outer measure]
If $E \subset F \subset \mathbb{R}^d$, then 
$\overline{\mu}(E) \leq \overline{\mu}(F)$. 
\end{prop}

\begin{proof}
For any open $\mathcal{O}$ with $E \subset F \subset \mathcal{O}$, 
we have $\chi_E(x) \leq \chi_F(x)$ on $\mathcal{O}$; 
the claim follows from the definition of the upper integral. 
\end{proof}

The next statement identifies our outer measure with the usual Lebesgue outer measure. 
\begin{prop}[Representation of the outer measure]
\label{outer measure1}
For any $E \subset \mathbb{R}^d$, 
\[
\overline{\mu}(E)=\inf
\{
\mu(\mathcal{O}): E \subset \mathcal{O} \in \mathcal{O}(\mathbb{R}^d)
\}. 
\]
\end{prop}

\begin{proof}
By definition, for any $\varepsilon>0$ there exists a partition $\{I_m\}$ with 
\[
\sum_{m=1}^{\infty}(\esssup_{x \in I_m}\chi_{E}(x))|I_m|<
\overline{\mu}(E)+\varepsilon. 
\]
Now 
\[
\esssup_{x \in I_m}\chi_{E}(x)=
\begin{cases}
1, & (E \cap I_m \notin \mathcal{N}(\mathbb{R}^d))\\[1ex]
0, & (E \cap I_m \in \mathcal{N}(\mathbb{R}^d))
\end{cases}
\]
and terms with value $0$ do not contribute. 
Considering only those $I_{m(k)}$ with 
$E \cap I_{m(k)} \notin \mathcal{N}(\mathbb{R}^d)$, we obtain 
\[
\sum_{m=1}^{\infty}(\esssup_{x \in I_m}\chi_{E}(x))|I_m|=
\sum_{k \in K}|I_{m(k)}|<\overline{\mu}(E)+\varepsilon. 
\]
where $K$ is at most countable. 
Set
\[
E\subset \mathcal{O}':=\bigsqcup_{k\in K}I_{m(k)}, 
\quad 
\mu(\mathcal{O}')=\sum_{k \in K}|I_{m(k)}|
\]
Since $\mathcal{O}'$ is open and $\varepsilon>0$ is arbitrary, 
\[
\inf\limits_{\substack{\mathcal{O} \supset E \\ 
\mathcal{O} \in \mathcal{O}(\mathbb{R}^d)}}\mu(\mathcal{O}) \leq \overline{\mu}(E). 
\]
The reverse inequality follows from monotonicity. 
\end{proof}

\begin{cor}\label{outer measure2}
For $E \subset \mathbb{R}^d$, 
\[
\overline{\mu}(E)=
\inf \left \{\sum_{m=1}^{\infty}|I_m| :
E \subset \bigsqcup_{m=1}^{\infty}I_m\text{ a.e.},\  I_m \in \mathscr{I}(\mathbb{R}^d)
\right \}. 
\]
\end{cor}

Since our outer measure agrees with the Lebesgue outer measure, 
the following is immediate.

\begin{lem}[$\sigma$-subadditivity]
For any sequence $\{E_n\}$ with $E_n \subset \mathbb{R}^d$, 
\[
\overline{\mu}\left(\bigcup_{n=1}^{\infty}E_n\right ) 
\leq \sum_{n=1}^{\infty}\overline{\mu}(E_n). 
\]
\end{lem}

\begin{proof}
Because 
$\overline{\mu}(E)= \inf\limits_{\mathcal{O} \supset E}\mu(\mathcal{O})$, 
for each $E_n$ and every $\varepsilon>0$ there is an open set 
$\mathcal{O}_n=\bigsqcup\limits_{m=1}^{\infty}I_{m,n}$ with 
\[
\mu(\mathcal{O}_n)=\sum_{m=1}^{\infty}|I_{m,n}| <\overline{\mu}(E_n)+\frac{\varepsilon}{2^n}. 
\]
Then 
\[
\sum_{n=1}^{\infty}\sum_{m=1}^{\infty}|I_{m,n}|
\leq \sum_{n=1}^{\infty}\overline{\mu}(E_n) + \varepsilon. 
\]
Let $\mathcal{O}':=\bigcup\limits_{n=1}^{\infty}\bigcup\limits_{m=1}^{\infty}I_{m,n}$, 
an open set containing 
$
\bigcup\limits_{n=1}^{\infty}E_n
$. 
By monotonicity, 
\[
\overline{\mu}\left(\bigcup_{n=1}^{\infty}E_n  \right )
\leq \overline{\mu}(\mathcal{O}')
\leq \sum_{n=1}^{\infty}\sum_{m=1}^{\infty}|I_{m,n}|
\leq \sum_{n=1}^{\infty}\overline{\mu}(E_n) + \varepsilon
\]
Since $\varepsilon>0$ is arbitrary, the claim follows.
\end{proof}

Whereas the outer measure was approximated from the outside by open sets, 
the inner measure admits approximation 
from the inside by closed (indeed, compact) sets.

\begin{prop}[Representation of the inner measure]\label{inner measure1}
For any $E \subset \mathbb{R}^d$, 
\[
\underline{\mu}(E)=\sup
\{
\overline{\mu}(F) : F \subset E\ \text{a.e.}, F \in \mathcal{F}(\mathbb{R}^d)
\}. 
\]

\begin{rem}
The \lq \lq a.e.\rq \rq \
qualifier can be omitted, 
but it is convenient in applications, so we retain it here
{\rm(}and likewise in Corollary \ref{inner measure2}{\rm)}. 
Later we will show 
$\mathcal{F}(\mathbb{R}^d) \subset \mathcal{M}(\mathbb{R}^d)$, 
permitting $\overline{\mu}$ to be replaced by $\mu$. 
\end{rem}

\begin{proof}
Let $\alpha_m:=\essinf\limits_{x \in I_m}\chi_E(x)$. 
By the definition of $\underline{\mu}(E)$, 
there exist $\{I_m\} \in  \Pi(\mathcal{O})$ with 
$E \subset \mathcal{O}$ such that  
\[
\underline{\mu}(E) -\frac{\varepsilon}{2} <
\sum_{m=1}^{\infty}\alpha_m |I_m|. 
\]
The right-hand side can be approximated by a finite sum:
\[
\sum_{m=1}^{\infty}\alpha_m |I_m|
-\sum_{m=1}^{N}\alpha_m |I_m|<\frac{\varepsilon}{2}
\]
Hence 
\[
\underline{\mu}(E) -\varepsilon < \sum_{m=1}^N \alpha_m |I_m|
= \sum_{k \in K}\alpha_{m(k)}|I_{m(k)}|\leq \underline{\mu}(E). 
\]
where 
\[
K:=\{k \in \{1, 2, \dots, N\}: \overline{I_{m(k)}} \subset E\ \ \text{a.e.}\}. 
\]
Set 
\[
F:=
\bigsqcup_{k \in K}\overline{I_{m(k)}}\subset E\ \ \text{a.e.}
\]
Then $F$ is compact, and we obtain
\[
\underline{\mu}(E)=\sup_{\substack{F \subset E \text{ a.e.} \\
F \in \mathcal{F}(\mathbb{R}^d)}}\overline{\mu}(F). 
\]
\end{proof}
\end{prop}

\begin{cor}\label{inner measure2}
For any $E \subset \mathbb{R}^d$, 
\[
\underline{\mu}(E)
= \sup
\left \{
\sum_{n=1}^N |I_n|: \bigsqcup_{n=1}^N \overline{I_n}\subset E \text{ a.e.}, I_n \in \mathscr{I}(\mathbb{R}^d)
\right \}. 
\]
\end{cor}

Although the integral was defined on open sets, 
we can also define it on finite-measure sets.

\begin{defn}[Integrals on finite-measure sets]
For $E \in \mathcal{M}(\mathbb{R}^d)$ and 
$f\in \mathcal{M}^1(\mathbb{R}^d)$, set 
\[ 
\int_{E}f(x)\,d_{\mathrm{M}}x
:=\int_{\mathbb{R}^d} f(x)\chi_{E}(x)\,d_{\mathrm{M}}x. 
\]
\end{defn}

\section{Essential Riemann Criterion for $\mathcal{M}^1$-Integrability
}

The following theorem gives an essential-oscillation (Riemann-type) criterion that characterizes $\mathcal{M}^1$-integrability.

\begin{thm}[Essential Riemann criterion for $\mathcal{M}^1$-integrability]
\label{Riemann conditions}
Let $\mathcal{O}$ be a open set in $\mathbb{R}^d$ and 
$f:\mathcal{O} \to \mathbb{R}$ with $f \geq 0$ a.e. on $\mathcal{O}$. 
The following are equivalent.
\begin{itemize}

\item[\rm{(i)}] $f \in \mathcal{M}^1(\mathcal{O})$.

\item[\rm{(ii)}] 
The upper integral of $f$ is finite, 
and for every $\varepsilon>0$ 
there exists a partition $\{I_m\} \in \Pi(\mathcal{O})$ 
such that 
\[
\sum_{m=1}^{\infty}(\beta_m-\alpha_m)|I_m|<\varepsilon, 
\]
where 
$\alpha_m=\essinf\limits_{x \in I_m}f(x)$ and $\beta_m=\esssup\limits_{x \in I_m}f(x)$.

\end{itemize}
\end{thm}

\begin{proof}
\noindent [\,(i)$\implies$(ii)\,]
Assume $f \in \mathcal{M}^1(\mathcal{O})$ and fix $\varepsilon>0$. 
By definition there exist $\{J^1_m\}$, $\{J^2_m\} \in \Pi(\mathcal{O})$ such that 
\[
0 \leq 
\int_{\mathcal{O}}f\,d_{\mathrm{M}}
-L(f,\{J^1_m\})<\frac{\varepsilon}{2}, 
\quad 
0 \leq 
U(f,\{J^2_m\})
-
\int_{\mathcal{O}}f\,d_{\mathrm{M}}<\frac{\varepsilon}{2}. 
\]
The common refinement 
$\{J^1_m \cap J^2_n \}_{m,n}$ lies in $\Pi(\mathcal{O})$, 
and by monotonicity under refinement and the defining inequalities,
\[
L(f,\{J^1_m\})
\leq 
L(f,\{J^1_m\cap J^2_n\}) 
\leq \int_{\mathcal{O}}f\,d_{\mathrm{M}}
\leq 
U(f,\{J^1_m\cap J^2_n\}) 
\leq 
U(f,\{J^2_n\}). 
\]
Hence 
\[ 
U(f,\{J^1_m\cap J^2_n\}) -
L(f,\{J^1_m\cap J^2_n\})<\varepsilon, \] 
which is exactly (ii).
\\[1ex]
\noindent [\,(ii)$\implies$(i)\,]
Given $\varepsilon>0$, 
condition (ii) and the definitions of the upper and lower integrals imply 
\[
\varepsilon>
\sum_{m=1}^{\infty} (\beta_m-\alpha_m)|I_m|
\geq 
\overline{\int_{\mathcal{O}}}f(x)\,d_{\mathrm{M}}x
-
\underline{\int_{\mathcal{O}}}f(x)\,d_{\mathrm{M}}x. 
\]
Since $\varepsilon>0$ is arbitrary, we obtain 
\[ \displaystyle 
\underline{\int_{\mathcal{O}}}f(x)\,d_{\mathrm{M}}x
= \overline{\int_{\mathcal{O}}}f(x)\,d_{\mathrm{M}}x<\infty
\] 
So $f \in \mathcal{M}^1(\mathcal{O})$. 
\end{proof}

\begin{cor}\label{monotone approach}
If $f:\mathcal{O} \to \mathbb{R}$ with 
$f \geq 0$ a.e. and $f \in \mathcal{M}^1(\mathcal{O})$, 
then there exist partitions $\{I^n_m\}\in \Pi(\mathcal{O})$ 
{\rm(}for each $n \in \mathbb{N}${\rm)} such that, as $n \to \infty$, 
\[
L(f,\{I^n_m\})
\uparrow \int_{\mathcal{O}}f\,d_{\mathrm{M}}, \quad 
U(f,\{I^n_m\})
\downarrow 
\int_{\mathcal{O}}f\,d_{\mathrm{M}}. 
\]
\end{cor}

\begin{rem}[Notation for monotone convergence]
In this paper, for sequences of real numbers or sets, the symbol 
\lq \lq \ $\uparrow$\rq \rq \ is used as follows.
\begin{itemize}
\item[{\rm(i)}] $a_n \uparrow a$ $\iff$ $a_n \le a_{n+1}$ for all $n\in\mathbb{N}$ and $\lim\limits_{n\to\infty} a_n = a$.
\item[{\rm(ii)}] $A_n \uparrow A$ $\iff$ $A_n \subset A_{n+1}$ for all $n\in\mathbb{N}$ and $A=\bigcup\limits_{n=1}^\infty A_n$.
\end{itemize}
The symbol \lq \lq \ $\downarrow$\rq \rq \ is defined analogously.
\end{rem}

\section{Properties of the Integral for Nonnegative Functions}

In this section 
we establish the basic properties of the integral: 
linearity and monotonicity.

\begin{lem}[Homogeneity]
If $f \in \mathcal{M}^1(\mathbb{R}^d)$ and 
$\alpha \geq 0$, then 
\[
\int_{\mathbb{R}^d} \alpha f(x)\,d_{\mathrm{M}}x
=\alpha \int_{\mathbb{R}^d} f(x)\,d_{\mathrm{M}}x. 
\]
\end{lem}

\begin{proof}
From 
$
\alpha L(f,\{I_m\})=
L(\alpha f, \{I_m\})
\leq 
U(\alpha f, \{I_m\})
=\alpha U(f,\{I_m\})
$, we obtain
\[
\alpha \underline{\int_{\mathbb{R}^d}}f(x)\,d_{\mathrm{M}}x 
=  \underline{\int_{\mathbb{R}^d}} \alpha f(x)\,d_{\mathrm{M}}x 
\leq \overline{\int_{\mathbb{R}^d}}\alpha f(x)\,d_{\mathrm{M}}x
= \alpha \overline{\int_{\mathbb{R}^d}}f(x)\,d_{\mathrm{M}}x
\]
If $f \in \mathcal{M}^1(\mathbb{R}^d)$, 
the lower and upper integrals coincide, giving the claim. 
\end{proof}

\begin{lem}[Additivity] 
If $f, g \in \mathcal{M}^1(\mathbb{R}^d)$, then
\[
\int_{\mathbb{R}^d} f(x) + g(x) \,d_{\mathrm{M}}x 
= \int_{\mathbb{R}^d} f(x)\,d_{\mathrm{M}}x
 + \int_{\mathbb{R}^d} g(x)\,d_{\mathrm{M}}x. 
\]

\end{lem}

\begin{proof}
For any $A \subset \mathbb{R}^d$, 
\begin{equation*} 
\begin{split} 
\essinf_{x \in A}f(x)+\essinf_{x \in A}g(x)
&\leq \essinf_{x \in A}(f+g)(x) \\
&\leq \esssup_{x \in A}(f+g)(x) 
\leq \esssup_{x \in A}f(x)+\esssup_{x \in A}g(x). 
\end{split}
\end{equation*}
Consequently, by definition of upper and lower integrals, we have 
\begin{equation*} 
\begin{split} 
\underline{\int_{{\mathbb{R}^d}}}f\,d_{\mathrm{M}}x 
+ \underline{\int_{{\mathbb{R}^d}}}g\,d_{\mathrm{M}}x 
& \leq \underline{\int_{{\mathbb{R}^d}}}(f+g)\,d_{\mathrm{M}}x \\
&\leq \overline{\int_{{\mathbb{R}^d}}}(f+g)\,d_{\mathrm{M}}x
\leq \overline{\int_{{\mathbb{R}^d}}}f\,d_{\mathrm{M}}x
+\overline{\int_{{\mathbb{R}^d}}}g\,d_{\mathrm{M}}x. 
\end{split}
\end{equation*}
Since $f, g \in \mathcal{M}^1(\mathbb{R}^d)$, 
the middle quantities coincide, proving the identity.
\end{proof}

\begin{lem}[Nonnegativity]
If $f \in \mathcal{M}^1(\mathbb{R}^d)$ and $f \geq 0$ a.e., 
then  
\[
\int_{\mathbb{R}^d} f(x) \,d_{\mathrm{M}}x \geq 0. 
\]
\end{lem}

\begin{proof}
Clearly $L(f, \{I_m\})\geq 0$, 
hence the integral 
is nonnegative.
\end{proof}

\begin{cor}[Monotonicity]
If $f, g \in \mathcal{M}^1({\mathbb{R}^d})$ and 
$f \leq g$ a.e., then 
\[ 
\int_{\mathbb{R}^d} f(x)\,d_{\mathrm{M}}x 
\leq \int_{\mathbb{R}^d} g(x)\,d_{\mathrm{M}}x. 
\] 
\end{cor}

\begin{cor}[Domain monotonicity]
Let $E, F \in \mathcal{M}(\mathbb{R}^d)$ and 
$f \in \mathcal{M}^1(\mathbb{R}^d)$ with $f \geq 0$. 
If $E \subset F$ a.e., then  
\[
\int_E f(x)\,d_{\mathrm{M}}x \leq \int_F f(x)\,d_{\mathrm{M}}x. 
\]
\end{cor}

\section{Measurable Sets $\mathcal{M}^*(\mathbb{R}^d)$}

A set $E$ has finite measure precisely 
when the upper and lower integrals of its characteristic function coincide. 
Under the hypothesis $\overline{\mu}(E)<\infty$,  
this is equivalent to the Carath\'odory condition:
\begin{quote}
For every $A \subset \mathbb{R}^d$,\quad 
$
\overline{\mu}(A)
= \overline{\mu}(A \cap E) + \overline{\mu}(A \cap E^c)
$. 
\end{quote}
In this section we prove this equivalence, 
and then use the condition to define 
\lq \lq measurable sets\rq \rq even when the measure is not finite.

\begin{lem}[Representation of inner measure by outer measure]
\label{representation of inner measure by outer measure}
Let $E \subset \mathcal{O}$ with $\mu(\mathcal{O})<\infty$. 
Then
\[ \underline{\mu}(E)=\mu(\mathcal{O})-\overline{\mu}(\mathcal{O}\cap E^c). \] 
\end{lem}

\begin{proof}
Fix $\{I_m\} \in \Pi(\mathcal{O})$. 
Since $\chi_{E}(x) +\chi_{\mathcal{O}\cap E^c}(x)=1$ on $\mathcal{O}$,  
\[
\chi_E(x) |I_m| = |I_m| -\chi_{\mathcal{O}\cap E^c}(x)|I_m|. 
\]
Taking essential infimum over $I_m$ yields 
\begin{equation*} 
\begin{split} 
(\essinf_{x \in I_m}\chi_E(x)) |I_m|
& =
\essinf_{x \in I_m}(
 |I_m| -\chi_{\mathcal{O}\cap E^c}(x)|I_m|)\\[1ex]
 &=
 |I_m|-(\esssup_{x \in I_m}\chi_{\mathcal{O}\cap E^c}(x))|I_m|
\end{split}
\end{equation*}
Summing over $m$ and then taking the supremum over 
$\{I_m\}\in  \Pi(\mathcal{O})$ gives the claim. 
\end{proof}

We next show that intersections of open sets are open, 
and use this to prove 
that arbitrary open sets (with finite measure) 
satisfy Carath\'odory's condition in Lemmas \ref{preCC4open} 
and \ref{CC4open}. 
Combining this with Lemmas \ref{relation inner and outer} and \ref{measurability4measurable+open}, 
we see that the intersection of a finite-measure set 
with an open set again has finite measure. 
Finally, Proposition \ref{CC} establishes the announced equivalence.

\begin{lem}\label{preCC4open}
For open sets 
$\mathcal{O}_1, \mathcal{O}_2$ with $\mu(\mathcal{O}_1)<\infty$, 
\[
\mu(\mathcal{O}_1)=\mu(\mathcal{O}_1 \cap \mathcal{O}_2)
+ \overline{\mu}(\mathcal{O}_1 \cap \mathcal{O}_2^c). 
\]
\end{lem}

\begin{proof}
Apply Lemma \ref{representation of inner measure by outer measure} with  
$\mathcal{O}=\mathcal{O}_1$ and $E=\mathcal{O}_1 \cap \mathcal{O}_2$. 
Since 
\[ 
E^c=(\mathcal{O}_1\cap \mathcal{O}_2)^c=\mathcal{O}_1^c \cup \mathcal{O}_2^c, 
\quad 
\mathcal{O}_1 \cap E^c
= \mathcal{O}_1 \cap \mathcal{O}_2^c, 
\]
we obtain 
\[
\underline{\mu}(\mathcal{O}_1 \cap \mathcal{O}_2)
=\mu(\mathcal{O}_1)- \overline{\mu}(\mathcal{O}_1 \cap \mathcal{O}_2^c). 
\]
Because $\mathcal{O}_1 \cap \mathcal{O}_2$ 
is an open set of finite measure, 
it belongs to $\mathcal{M}(\mathbb{R}^d)$, and the identity yields the lemma.
\end{proof}

\begin{lem}\label{CC4open}
For any 
$E\subset \mathbb{R}^d$ and any open $\mathcal{O}\subset \mathbb{R}^d$, 
\[
\overline{\mu}(E)=\overline{\mu}(E \cap \mathcal{O})
+ \overline{\mu}(E \cap \mathcal{O}^c). 
\]
\end{lem}

\begin{proof}
Since $\overline{\mu}(E)$ 
can be approximated from outside by open sets, for any $\varepsilon>0$ there exists 
$\mathcal{O}_1 \supset E$ with 
\[
\mu(\mathcal{O}_1) \leq \overline{\mu}(E) + \varepsilon. 
\]
By Lemma \ref{preCC4open} and monotonicity of $\overline{\mu}$, 
\begin{equation*} 
\begin{split} 
\mu(\mathcal{O}_1)
= \mu(\mathcal{O}_1 \cap \mathcal{O})+ \overline{\mu}(\mathcal{O}_1 \cap \mathcal{O}^c) 
\geq \overline{\mu}(E \cap \mathcal{O}) + \overline{\mu}(E \cap \mathcal{O}^c). 
\end{split}
\end{equation*}
Thus 
\[
\overline{\mu}(E) + \varepsilon \geq \mu(\mathcal{O}_1)
\geq 
\overline{\mu}(E \cap \mathcal{O}) + \overline{\mu}(E \cap \mathcal{O}^c). 
\]
Since 
$\varepsilon>0$
is arbitrary and the reverse inequality follows from 
$\sigma$-subadditivity, the lemma holds. 
\end{proof}

\begin{lem}\label{relation inner and outer}
For any $A, B \subset \mathbb{R}^d$, 
\[
\underline{\mu}(A \sqcup B)
\leq \underline{\mu}(A) + \overline{\mu}(B) \leq \overline{\mu}(A \sqcup B). 
\]
\end{lem}

\begin{proof}
For any functions $f$, $g$ and a set $A \subset \mathbb{R}^d$, 
\[
\essinf_{x \in A} (f(x)+g(x)) 
\leq \essinf_{x \in A}f(x) + \esssup_{x \in A}g(x) \leq \esssup_{x \in A}(f(x)+g(x))
\]
Apply this with $f = \chi_A$, $g=\chi_B$ 
and take the appropriate supremum/infimum over partitions.
\end{proof}

\begin{lem}\label{measurability4measurable+open}
If $E \in \mathcal{M}(\mathbb{R}^d)$, then for every open $\mathcal{O}$
we have $E\cap \mathcal{O} \in \mathcal{M}(\mathbb{R}^d)$. 
\end{lem}

\begin{proof}
From Lemma \ref{relation inner and outer} and 
Lemma \ref{CC4open}, we have 
\begin{equation*} 
\begin{split} 
\underline{\mu}(E) 
& =\underline{\mu}((E \cap \mathcal{O})\sqcup (E \cap \mathcal{O}^c)) \\
& \leq 
\underline{\mu}(E \cap \mathcal{O})
+\overline{\mu}(E \cap \mathcal{O}^c)
\leq \overline{\mu}(E) 
= \overline{\mu}(E\cap \mathcal{O}) 
+ \overline{\mu}(E \cap \mathcal{O}^c). 
\end{split}
\end{equation*}
Hence 
$
\underline{\mu}(E \cap \mathcal{O})
=\overline{\mu}(E\cap \mathcal{O}) 
$, so 
$E \cap \mathcal{O} \in \mathcal{M}(\mathbb{R}^d)$. 
\end{proof}

\begin{prop}\label{CC} 
Let $E\subset\mathbb{R}^d$ with $\overline{\mu}(E)<\infty$.
Then $E\in\mathcal{M}(\mathbb{R}^d)$ 
if and only if for every $A \subset \mathbb{R}^d$,
\[
\overline{\mu}(A)
= \overline{\mu}(A\cap E) + \overline{\mu}(A\cap E^c). 
\]
\end{prop}

\begin{proof}
First assume the displayed equality holds. 
For any open $\mathcal{O} \supset E$ with $\mu(\mathcal{O})<\infty$, 
\[
\overline{\mu}(\mathcal{O})=\overline{\mu}(E) + \overline{\mu}(\mathcal{O}\cap E^c). 
\]
i.e., 
$\overline{\mu}(\mathcal{O})-\overline{\mu}(\mathcal{O}\cap E^c)=\overline{\mu}(E)$. 
By Lemma \ref{representation of inner measure by outer measure}, 
the left-hand side equals 
$\underline{\mu}(E)$; 
thus $\underline{\mu}(E)=\overline{\mu}(E)$. 

Conversely, for any $A$ we choose open $\mathcal{O} \supset A$. 
Apply Lemma \ref{representation of inner measure by outer measure} with
$E' := \mathcal{O}\cap E$, so $(E')^c= \mathcal{O}^c \cup E^c$, 
then 
\begin{equation*} 
\begin{split} 
\underline{\mu}(\mathcal{O} \cap E) 
=\mu(\mathcal{O}) -\overline{\mu}(\mathcal{O}\cap E^c). 
\end{split}
\end{equation*}
Since $E \in \mathcal{M}(\mathbb{R}^d)$, Lemma \ref{measurability4measurable+open} 
yields $\mathcal{O}\cap E \in \mathcal{M}(\mathbb{R}^d)$; 
hence $\underline{\mu}(\mathcal{O} \cap E) =\overline{\mu}(\mathcal{O} \cap E)$. 
Consequently, we have 
\begin{equation*} 
\begin{split} 
\mu(\mathcal{O}) 
 = \underline{\mu}(\mathcal{O} \cap E) 
+ \overline{\mu}(\mathcal{O}\cap E^c) 
= \overline{\mu}(\mathcal{O} \cap E) 
+ \overline{\mu}(\mathcal{O}\cap E^c)
\geq 
\overline{\mu}(A \cap E) 
+ \overline{\mu}(A\cap E^c)
\end{split}
\end{equation*}
by monotonicity. 
Taking the infimum over open $\mathcal{O} \supset A$ yields 
\[
\overline{\mu}(A)
\geq \overline{\mu}(A\cap E) + \overline{\mu}(A\cap E^c), 
\]
and the reverse inequality follows from 
$\sigma$-subadditivity. This is the Carath\'eodory condition. 
\end{proof}

Thus, within the realm of finite measures, 
finite-measure sets coincide with sets satisfying Carath\'eodory's condition. 
However, finite-measure sets only cover those sets of finite measure, 
whereas the Carath\'eodory class imposes no finiteness requirement 
and hence includes sets of infinite measure. 
We therefore extend the class as follows. 

\begin{defn}[Measurable sets $\mathcal{M}^*(\mathbb{R}^d)$]
We say $E \in \mathcal{M}^*(\mathbb{R}^d)$ if for every 
$A \subset \mathbb{R}^d$, 
\[
\overline{\mu}(A)
= \overline{\mu}(A\cap E) + \overline{\mu}(A\cap E^c). 
\]
An element of $\mathcal{M}^*(\mathbb{R}^d)$ 
is called a \emph{measurable set}. In particular, if 
$E \in \mathcal{M}^*(\mathbb{R}^d)$ but $E \notin \mathcal{M}(\mathbb{R}^d)$, 
we define 
\[
\mu(E)=\underline{\mu}(E)=\overline{\mu}(E)=\infty. 
\]
\end{defn}

\begin{defn}
The pair $(\mathbb{R}^d, \mathcal{M}^*(\mathbb{R}^d))$
is called a \emph{measurable space}; together with a measure $\mu$ 
the triple $(\mathbb{R}^d, \mathcal{M}^*(\mathbb{R}^d), \mu)$  
is called a \emph{measure space}. 
\end{defn}

\section{$\sigma$-Algebra Structure and $\sigma$-Additivity}

In this section we recall that the collection of measurable sets forms a 
$\sigma$-algebra, i.e., 
it is closed under the standard set-theoretic operations. 
This is well known in the classical Lebesgue theory 
and is not an original contribution; 
we include it to keep the exposition self-contained 
and to clarify the overall structure of the theory. 
The material here essentially follows 
\cite[Theorem 1.17 and Proposition 1.6]{a-dp-m}, 
\cite[Theorem 1.3.6 and Proposition 1.2.5]{c}, 
\cite[Theorem 1.11 and Theorem 1.8]{Folland},  
\cite[\S 2.3, Proposition 13 and Theorem 15]{r-f},  
\cite[Theorem 3.24 and Proposition 3.30 and Proposition 3.34]{s-k},  
\cite[Chapter 1: Theorem 3.2 and Corollary 3.3; Chapter 6: Theorem 1.1]{s-s}

\begin{defn}[$\sigma$-algebra]
Let $X$ be a set. 
A family  $\mathcal{E}$ of subsets of $X$ 
is a $\sigma$-algebra if: 

\begin{itemize}

\item[\rm{(i)}] $X \in \mathcal{E}$ and $\varnothing \in \mathcal{E}$.

\item[\rm{(ii)}] If $E \in \mathcal{E}$, then $E^c \in \mathcal{E}$.

\item[\rm{(iii)}] If $\{E_n\} \subset \mathcal{E}$, then 
$\bigcup\limits_{n=1}^{\infty}E_n \in \mathcal{E}$.

\end{itemize}
\end{defn}

\begin{prop}\label{sigma-algebra}
$\mathcal{M}^*(\mathbb{R}^d)$ is a $\sigma$-algebra.  
\end{prop}

We prove this by a sequence of lemmas.

\begin{lem}[Measurability of complements]\label{measurability4complement}
If $E \in \mathcal{M}^*(\mathbb{R}^d)$, then 
$E^c \in \mathcal{M}^*(\mathbb{R}^d)$. 
\end{lem}

\begin{proof}
Immediate from Carath\'eodory's condition.
\end{proof}

\begin{lem}[Measurability of null sets]\label{measurability4empty}
If $E \in \mathcal{N}(\mathbb{R}^d)$, then $E \in \mathcal{M}^*(\mathbb{R}^d)$. 
\end{lem}

\begin{proof}
Since 
$\overline{\mu}(E)=\underline{\mu}(E)=0$, we have 
$E \in \mathcal{M}(\mathbb{R}^d)
\subset \mathcal{M}^*(\mathbb{R}^d)$. 
\end{proof}

\begin{cor}[Measurability of $\varnothing$ and $\mathbb{R}^d$]
$\varnothing, 
\mathbb{R}^d \in \mathcal{M}^*(\mathbb{R}^d)$. 
\end{cor}

\begin{proof}
Because $\overline{\mu}(\varnothing)=0$, we have  
$\varnothing \in \mathcal{N}(\mathbb{R}^d)$. 
Now apply Lemmas \ref{measurability4complement} and \ref{measurability4empty}.
\end{proof}

\begin{lem}
If $E_1, E_2 \in \mathcal{M}^*(\mathbb{R}^d)$, then 
$E_1 \cap E_2 \in \mathcal{M}^*(\mathbb{R}^d)$. 
\end{lem}

\begin{proof}
Using measurability of $E_1$ and $E_2$ successively, 
\begin{equation*} 
\begin{split} 
\overline{\mu}(A) 
& = \overline{\mu}(A \cap E_1)+\overline{\mu}(A \cap E_1^c) \\[1ex]
& = \overline{\mu}(A \cap E_1 \cap E_2)
+\overline{\mu}(A \cap E_1 \cap E_2^c) 
+ 
\overline{\mu}(A \cap E_1^c \cap E_2)
+\overline{\mu}(A \cap E_1^c \cap E_2^c). 
\end{split}
\end{equation*}
Since 
\begin{equation*} 
\begin{split} 
(E_1 \cap E_2^c) \cup (E_1^c \cap E_2) \cup (E_1^c \cap E_2^c)
= (E_1 \cap E_2)^c, 
\end{split}
\end{equation*}
$\sigma$-subadditivity of $\overline{\mu}$ yields
\[
\overline{\mu}(A \cap E_1 \cap E_2^c)
+
\overline{\mu}(A \cap E_1^c \cap E_2)
+
\overline{\mu}(A \cap E_1^c \cap E_2^c)
\geq \overline{\mu}(A\cap (E_1 \cap E_2)^c). 
\]
Hence
\[
\overline{\mu}(A) 
\geq \overline{\mu}(A \cap (E_1 \cap E_2))
+ \overline{\mu}(A\cap (E_1 \cap E_2)^c). 
\]
The reverse inequality follows from $\sigma$-subadditivity, proving measurability.  
\end{proof}

\begin{cor}\label{cupmeasurability}
If $E_1, E_2 \in \mathcal{M}^*(\mathbb{R}^d)$, then 
$E_1 \cup E_2 \in \mathcal{M}^*(\mathbb{R}^d)$. 
\end{cor}

\begin{lem} \label{additivity} 
Let 
$E_1, E_2 \subset \mathbb{R}^d$ with 
$E_1 \cap E_2=\varnothing$. 
If 
$E_1  \in \mathcal{M}^*(\mathbb{R}^d)$
{\rm(}or $E_2 \in \mathcal{M}^*(\mathbb{R}^d)${\rm)}, then 
\[
\overline{\mu}(E_1 \sqcup E_2) 
= \overline{\mu}(E_1) + 
\overline{\mu}(E_2) 
\]
\end{lem}

\begin{proof}
Assuming 
$E_1 \in \mathcal{M}^*(\mathbb{R}^d)$, Carath\'eodory's condition gives
\begin{equation*} 
\begin{split} 
\overline{\mu}(E_1 \sqcup E_2) 
& = \overline{\mu}((E_1 \sqcup E_2 )\cap E_1)+
\overline{\mu}((E_1 \sqcup E_2) \cap E_1^c)
  = 
\overline{\mu}(E_1)+
\overline{\mu}(E_2). 
\end{split}
\end{equation*}
\end{proof}

\begin{lem}\label{finite additivity}
If $\{E_n\} \subset \mathcal{M}^*(\mathbb{R}^d)$ 
with 
$E_m \cap E_n =\varnothing$ for $m \neq n$, 
then for any $A \subset \mathbb{R}^d$ and any $N \in \mathbb{N}$, 
\[
\overline{\mu}(A \cap \bigsqcup_{n=1}^N E_n)= \sum_{n=1}^N \overline{\mu}(A \cap E_n). 
\]
\end{lem}

\begin{proof}
Trivial for $N=1$. 
Assume the identity for $n=k$: i.e. 
$
\overline{\mu}(A \cap \bigsqcup\limits_{n=1}^k E_n)
= \sum\limits_{n=1}^k \overline{\mu}(A \cap E_n)
$. 
Using measurability of $E_{k+1}$ 
and the identity for $n=k$, 
\begin{equation*} 
\begin{split} 
\overline{\mu}(A \cap \bigsqcup\limits_{n=1}^{k+1} E_n)
&=\overline{\mu}((A \cap \bigsqcup\limits_{n=1}^{k+1} E_n)\cap E_{k+1})
+ 
 \overline{\mu}(A \cap \bigsqcup\limits_{n=1}^{k+1} E_n \cap E_{k+1}^c) \\
&=\overline{\mu}(A \cap E_{k+1})
+ 
 \overline{\mu}(A \cap \bigsqcup\limits_{n=1}^{k} E_n) \\
 & =\overline{\mu}(A \cap E_{k+1})
 + \sum_{n=1}^k\overline{\mu}(A \cap E_n)= \sum_{n=1}^{k+1}\overline{\mu}(A \cap E_n). 
\end{split}
\end{equation*}
Therefore, it also holds for $n=k+1$. Hence the lemma holds for any $N \in \mathbb{N}$. 
\end{proof}

\begin{lem} 
If 
$\{E_n\} \subset \mathcal{M}^*(\mathbb{R}^d)$ with 
$E_m \cap E_n = \varnothing$ for $m\neq n$, 
then $\bigsqcup\limits_{n=1}^{\infty}E_n \in \mathcal{M}^*(\mathbb{R}^d)$. 
\end{lem}

\begin{proof}
By Corollary \ref{cupmeasurability} and Carath\'eodory's condition, 
for any $A \subset \mathbb{R}^d$ and $N \in \mathbb{N}$, 
\begin{equation*} 
\begin{split} 
\overline{\mu}(A) 
&
= 
\overline{\mu}\left (A\cap \Bigl (\bigsqcup_{n=1}^NE_n \Bigr ) \right )
+\overline{\mu}\left (A \cap \Bigl (\bigsqcup_{n=1}^NE_n \Bigr )^c\right ) \\ 
& \geq 
\sum_{n=1}^N
\overline{\mu}\left (A\cap E_n \right )
+\overline{\mu}\left (A \cap \Bigl (\bigsqcup_{n=1}^{\infty}E_n \Bigr )^c\right )
\end{split}
\end{equation*}
where we used monotonicity and Lemma \ref{finite additivity}. 
Letting $N \to \infty$ and using $\sigma$-subadditivity gives 
\begin{equation*} 
\begin{split} 
\overline{\mu}(A) 
&
\geq 
\sum_{n=1}^{\infty}\overline{\mu}(A\cap E_n)
+\overline{\mu}\left (A \cap \Bigl (\bigsqcup_{n=1}^{\infty}E_n \Bigr )^c\right ) \\
& \geq 
\overline{\mu}\left (A\cap \Bigl (\bigsqcup_{n=1}^{\infty}E_n \Bigr ) \right )
+\overline{\mu}\left (A \cap \Bigl (\bigsqcup_{n=1}^{\infty}E_n \Bigr )^c\right ), 
\end{split}
\end{equation*}
and the reverse inequality follows from subadditivity. Hence 
$\bigsqcup\limits_{n=1}^{\infty}E_n \in \mathcal{M}^*(\mathbb{R}^d)$. 
\end{proof}

\begin{lem}
If $\{E_n\} \subset \mathcal{M}^*(\mathbb{R}^d)$, 
then $\bigcup\limits_{n=1}^{\infty}E_n \in \mathcal{M}^*(\mathbb{R}^d)$. 
\end{lem}

\begin{proof}
Define
\begin{equation*} 
F_1:= E_1, \qquad 
F_{n}:=E_n \cap \Bigl ({\bigcup\limits_{k=1}^{n-1}E_k} \Bigr )^c
\quad (n \geq 2). 
\end{equation*}
Then $\{F_n\}$ are pairwise disjoint measurable sets and 
\[
\bigcup_{n=1}^{\infty}E_n=
\bigsqcup_{n=1}^{\infty}F_n \in \mathcal{M}^*(\mathbb{R}^d). 
\]
\end{proof}

\begin{prop}[$\sigma$-additivity of $\mu$]
\label{sigma}
If 
$\{E_n\} \subset \mathcal{M}^*(\mathbb{R}^d)$ with 
$E_m \cap E_n = \varnothing\ (m\neq n)$, then 
\[
\mu\left (\bigsqcup_{n=1}^{\infty}E_n \right )=\sum_{n=1}^{\infty}\mu(E_n)
\]
\end{prop}

\begin{proof}
For any $N \in \mathbb{N}$, 
by monotonicity, 
Corollary \ref{cupmeasurability} and Lemma \ref{additivity}, 
\begin{equation*} 
\begin{split} 
\overline{\mu}\Bigl (\bigsqcup_{n=1}^{\infty}E_n \Bigr)
\geq \overline{\mu} \Bigl (\bigsqcup_{n=1}^{N}E_n \Bigr )
= \sum_{n=1}^{N}\overline{\mu}(E_n). 
\end{split}
\end{equation*}
Letting $N \to \infty$, we have 
\[
\overline{\mu}\Bigl (\bigsqcup_{n=1}^{\infty}E_n \Bigr)
\geq 
\sum_{n=1}^{\infty}\overline{\mu}(E_n).
\]
the reverse inequality follows from $\sigma$-subadditivity. 
Since all the sets involved are measurable, $\overline{\mu}=\mu$. 
\end{proof}

\begin{prop} \label{continuity}
Let $\{E_n\} \subset \mathcal{M}^*(\mathbb{R}^d)$ with $E_n \uparrow E$. 
Then 
\[
\mu\left (\bigcup_{n=1}^{\infty}E_n \right )=\lim_{n \to \infty}\mu(E_n). 
\]
\end{prop}

\begin{proof}
Set 
\begin{equation*} 
\begin{cases}
F_1:= E_1 \\[1ex]
F_{n}:=E_n \cap E_{n-1}^c & (n \geq 2). 
\end{cases}
\end{equation*}
Then 
$\{F_n\}$ are disjoint and Proposition \ref{sigma} yields 
\begin{equation*} 
\begin{split} 
{\mu}\left (\bigcup_{n=1}^{\infty}E_n \right ) 
& = {\mu}\left (\bigsqcup_{n=1}^{\infty}F_n \right ) 
=\sum_{n=1}^{\infty}\overline{\mu}(F_n) \\
& = {\mu}(F_1)+\sum_{n=2}^{\infty}{\mu}(E_n \cap E_{n-1}^c) \\
& = \lim_{N \to \infty}
\left ( 
{\mu}(E_1)+\sum_{n=2}^{N}
\Bigl ( {\mu}(E_n)-{\mu}( E_{n-1}) \Bigr )
\right )= \lim_{N \to \infty}{\mu}(E_N). 
\end{split}
\end{equation*}
\end{proof}

\section{Structure Theorem for Measurable Sets}

In Lebesgue measure theory it is well known that every Lebesgue measurable set can be written as a Borel set together with a Lebesgue null set. For completeness, we record the standard statements here; see, e.g., \cite[Theorem 1.19]{Folland}, 
\cite[Theorem 11]{r-f}, 
\cite[Theorem 3.36]{s-k}, 
\cite[Corollary 3.5]{s-s}.

\begin{defn}[$\sigma$-algebra generated by a family]
For a family $\mathcal{A}$ of sets, 
$\sigma(\mathcal{A})$ denotes the smallest $\sigma$-algebra containing
$\mathcal{A}:$
\[ \sigma(\mathcal{A}) 
=\bigcap\{\mathcal{F} \ | \ \mathcal{A} \subset \mathcal{F}, \ 
\text{$\mathcal{F}$ a $\sigma$-algebra}\}. \]

\end{defn}

\begin{defn}[Borel $\sigma$-algebra]
Let $\mathcal{O}(\mathbb{R}^d)$ be the open sets 
and $\mathcal{F}(\mathbb{R}^d)$ the closed sets. 
Define 
\[
\mathcal{B}(\mathbb{R}^d):=\sigma(\mathcal{O}(\mathbb{R}^d))
=\sigma(\mathcal{F}(\mathbb{R}^d)). 
\]
Elements of 
$\mathcal{B}(\mathbb{R}^d)$ are called \emph{Borel sets}. 
\end{defn}

\begin{defn}[$G_{\delta}$ and $F_{\sigma}$ sets]
A set $G$ is called 
a $G_{\delta}$ if it can be written 
as a countable intersection of open sets{\rm:} 
\[
G=\bigcap\limits_{n=1}^{\infty}\mathcal{O}_n, \ 
\mathcal{O}_n \in \mathcal{O}(\mathbb{R}^d). 
\]
A set $F$ is called an $F_{\sigma}$ 
if it is a countable union of closed sets{\rm:}
\[
F=\bigcup\limits_{n=1}^{\infty}F_n,  \ 
F_n \in \mathcal{F}(\mathbb{R}^d). 
\]
\end{defn}

\begin{lem}\label{regu}
If $E \in \mathcal{M}(\mathbb{R}^d)$, 
then there exist a $G_{\delta}$ set $G$ and a null set 
$\mathcal{N}$ with 
\[
E=G \cap \mathcal{N}^c
\]
\end{lem}

\begin{proof}
Assume $\mu(E)<\infty$. 
Since $E$ can be approximated from outside by open sets, 
for each $n \in \mathbb{N}$ there exists 
$\mathcal{O}_n \in \mathcal{O}(\mathbb{R}^d)$ such that  
\[ 
E \subset  \mathcal{O}_n,\quad 
  \mu(\mathcal{O}_n) < \mu(E) + \cfrac{1}{n}. \]
Let 
$
G:=\bigcap\limits_{n=1}^{\infty} \mathcal{O}_n
$. Then,  
since $E\in\mathcal{M}(\mathbb{R}^d)$ and $G$ is measurable, Carath\'eodory's condition, together with monotonicity and $E\subset G$, gives
\[ \mu(G \cap E^c) 
 = \mu(G) - \mu(G\cap E) \leq \mu(\mathcal{O}_n) -\mu(E) < \cfrac{1}{n}.  \]
Since $n$ is arbitrary,  
$\mathcal{N}:=G \cap E^c$ is null, and 
\[ E = G \cap E = 
G \cap (G^c \cup E)=G\cap (G\cap E^c)^c = G\cap \mathcal{N}^c. \]
\end{proof}

\begin{cor}
If $E \in \mathcal{M}(\mathbb{R}^d)$, 
then there exist an $F_{\sigma}$ set $F$
and a null set $\mathcal{N}'$ such that 
\[
E=F \cup \mathcal{N}'. 
\]
\end{cor}

\begin{proof}
Since $E$ is measurable, so is 
$E^c$; 
applying Lemma \ref{regu} to $E^c$ gives $E^c = G \cap (\mathcal{N}')^c$.
Hence 
\[
E= (G \cap (\mathcal{N}')^c)^c=G^c \cup \mathcal{N}'. 
\]
If $G=\bigcap\limits_{n=1}^{\infty}\mathcal{O}_n$, 
then by De Morgan, 
$G^c = \left (\bigcap\limits_{n=1}^{\infty}\mathcal{O}_n \right )^c
= \bigcup\limits_{n=1}^{\infty}\mathcal{O}_n^c$
, a countable union of closed sets, 
hence $G^c$ is $F_{\sigma}$. 
\end{proof}

\begin{prop}
\[
\mathcal{M}^*(\mathbb{R}^d)= \sigma\bigl(\mathcal{B}(\mathbb{R}^d) \cup \mathcal{N}(\mathbb{R}^d)\bigr).
\]
\end{prop}

\begin{proof}
Since 
$\mathcal{B}(\mathbb{R}^d)\subset \mathcal{M}^*(\mathbb{R}^d)$ and 
$\mathcal{N}(\mathbb{R}^d)\subset \mathcal{M}^*(\mathbb{R}^d)$, 
\[
\sigma\bigl(\mathcal{B}(\mathbb{R}^d)\cup\mathcal{N}(\mathbb{R}^d)\bigr)
\subset \mathcal{M}^*(\mathbb{R}^d).
\]

Conversely, if $E \in \mathcal{M}^*(\mathbb{R}^d)$, then  
\[ E=\bigcup_{k=1}^\infty (E\cap \{|x| \leq k\}).  \]
Each $E\cap \{|x| \leq k\}$ lies in $\mathcal{M}(\mathbb{R}^d)$, 
so by Lemma \ref{regu},
\[
E\cap \{|x| \leq k\}
=G_k\cap {\mathcal{N}}_k^c\quad(G_k\text{ a } G_\delta,\ \mathcal{N}_k\in\mathcal{N}(\mathbb{R}^d)). 
\]
Hence $E\cap \{|x|\leq k\}\in 
\sigma\bigl(\mathcal{B}(\mathbb{R}^d)\cup\mathcal{N}(\mathbb{R}^d)\bigr)$, 
and 
by the $\sigma$-algebra property,  
$E \in \sigma\bigl(\mathcal{B}(\mathbb{R}^d)\cup\mathcal{N}(\mathbb{R}^d)\bigr)$.
Therefore, 
$\mathcal{M}^*(\mathbb{R}^d) \subset 
\sigma\bigl(\mathcal{B}(\mathbb{R}^d)\cup\mathcal{N}(\mathbb{R}^d)\bigr)$.
\end{proof}

\section{Step-Function Approximation}

Although the classical Lebesgue integral is defined via approximation by simple functions, our definition did not adopt this as a starting point. Nevertheless, the corresponding approximation property follows from the behavior of the lower integral. 
We first recall the standard notions of convergence used in Lebesgue theory.

\begin{defn}[$\mathcal{M}^1$-convergence] 
Let 
$\{f_n\} \subset \mathcal{M}^1(\mathbb{R}^d)$ and 
$f \in \mathcal{M}^1(\mathbb{R}^d)$. 
We say
$\{f_n\}$ 
\emph{converges to $f$ in $\mathcal{M}^1$} if 
\[
\int_{{\mathbb{R}^d}}|f(x)-f_n(x)|\,d_{\mathrm{M}}x \to 0\quad (n \to \infty)
\]
and write 
$f_n \to f \text{\ \rm in }\mathcal{M}^1(\mathbb{R}^d)$. 
\end{defn}

\begin{defn}[Convergence in measure] 
For each $n \in \mathbb{N}$, let 
$f_n: \mathbb{R}^d \to \mathbb{R}$ and 
$f:\mathbb{R}^d \to \mathbb{R}$. 
We say \emph{$\{f_n\}$ converges to $f$ in measure} if 
\[
\overline{\mu}(\{x\in {\mathbb{R}^d} : |f(x)-f_n(x)|
\geq \varepsilon\})\to 0\quad (n \to \infty)
\]
for every $\varepsilon>0$, and write 
$
f_n \to f \text{\ \rm in measure }
$. 
\end{defn}

\begin{defn}[Almost-everywhere convergence]
For each $n \in \mathbb{N}$, let 
$f_n: \mathbb{R}^d \to \mathbb{R}$ and 
$f:\mathbb{R}^d \to \mathbb{R}$. 
If there exists 
$\mathcal{N} \in \mathcal{N}({\mathbb{R}^d})$ such that 
for all $x \in \mathcal{N}^c$, 
\[
f_n(x) \to f(x)\quad (n \to \infty), 
\]
then $\{f_n\}$ \emph{converges to $f$ 
almost everywhere}, 
written
\[
f_n \to f \text{\ \rm{a.e.}}\quad \text{or} \quad
f_n \to f \text{\ \rm{a.e.}}\ \mathbb{R}^d. 
\]
\end{defn}

\begin{defn}[Simple and step functions]
Consider functions $f:\mathbb{R}^d \to \mathbb{R}$ of the form
\[
f(x):=\sum_{m=1}^N \alpha_m \chi_{E_m}(x), 
\qquad N\in\mathbb{N},\ \alpha_m\in\mathbb{R}.
\]
If $E_m\in \mathcal{M}^*(\mathbb{R}^d)$ for each $m$, we call $f$ a \emph{simple function}. 
If, in addition, $E_m = I_m \in \mathscr{I}(\mathbb{R}^d)$ for each $m$, we call $f$ a \emph{step function}.
\end{defn}

\begin{prop}[Approximation by step functions]
\label{approximation by step functions}
Let $f \in \mathcal{M}^1(\mathbb{R}^d)$ with $f \geq 0$. 
Then there exists a monotonically increasing sequence of 
nonnegative step functions 
$\{\varphi_n\}$ such that,
as $n \to \infty$, 
\[
\varphi_n \uparrow f\text{\ \rm a.e. on }{\mathbb{R}^d}, \quad 
\varphi_n \to f\ \text{\ \rm in }\mathcal{M}^1(\mathbb{R}^d). 
\]
\end{prop}

We prove the proposition through several lemmas.

\begin{lem}[Existence of step functions converging in 
$\mathcal{M}^1$]
If $f \in \mathcal{M}^1({\mathbb{R}^d})$ with $f \geq 0$, then there exists a monotonically increasing sequence of 
step functions $\{\varphi_n\}$ with 
$\varphi_n \to f$ {\rm in} $\mathcal{M}^1(\mathbb{R}^d)$.  
\end{lem}

\begin{proof}
By the definition of the lower integral, for each $n \in \mathbb{N}$
there exists a partition $\{I^n_m\}_{m=1}^{\infty} \in \Pi({\mathbb{R}^d})$ and 
an integer $N(n) \in \mathbb{N}$
such that (cf.\ the proof of Proposition \ref{inner measure1} if needed) 
\[
\underline{\int_{\mathbb{R}^d}} f(x)\,d_{\mathrm{M}}x -\frac{1}{n} 
\leq \sum_{m=1}^{N(n)}\alpha_m^n|I_m^n| 
\leq  \underline{\int_{\mathbb{R}^d}} f(x)\,d_{\mathrm{M}}x,
\]
where $\alpha_m^n:=\essinf\limits_{x \in I_m^n} f(x)$, and 
\[
\sum_{m=1}^{N(n)}\alpha_m^n|I_m^n|
= \int_{{\mathbb{R}^d}} \sum_{m=1}^{N(n)} \alpha_m^n\chi_{I_m^n}(x)\,d_{\mathrm{M}}x.  
\]
Setting 
$\varphi_n:=\sum\limits_{m=1}^{N(n)}\alpha_m^n\chi_{I_m^n}$
yields the convergence in $\mathcal{M}^1(\mathbb{R}^d)$. 

Moreover, by choosing the partitions $\{I_m^n\}$ as in 
Corollary \ref{monotone approach} so that
\[ 
L(f,\{I_m^n\})\uparrow \int_{\mathbb{R}^d} f\,d_{\mathrm{M}}
\]
and arranging (e.g.\ by passing to common refinements) that each $\{I_m^{n+1}\}$ refines $\{I_m^{n}\}$, the associated step functions
$\varphi_n(x)=\sum\limits_{m=1}^{N(n)}\alpha_m^n\chi_{I_m^n}(x)$ satisfy
$\varphi_n \le \varphi_{n+1}$ a.e., hence are monotonically increasing.
\end{proof}

\begin{lem}
If $f_n \to f$ in $\mathcal{M}^1(\mathbb{R}^d)$, 
then $f_n \to f$ {\rm in measure}. 
\end{lem}

\begin{proof}
Fix $\varepsilon>0$ and set  
\[
E_{n}:=\{x \in {\mathbb{R}^d} \ :\ |f(x)-f_n(x)|> \varepsilon\}
\]
For all $x \in \mathbb{R}^d$, 
\[
|f(x)-f_n(x)| 
\geq \chi_{E_n}(x) \cdot |f(x)-f_n(x)|. 
\]
Taking upper integrals,
\begin{equation*}
\begin{split}
\int_{\mathbb{R}^d} |f-f_n|\,d_{\mathrm{M}} 
=\overline{\int_{\mathbb{R}^d}} |f-f_n|\,d_{\mathrm{M}}
 \geq \overline{\int_{\mathbb{R}^d}} |f-f_n|\chi_{E_{n}}\,d_{\mathrm{M}} 
 \geq \varepsilon \overline{\mu}(E_{n}). 
\end{split}
\end{equation*}
Hence
\begin{equation*}
\begin{split}
\overline{\mu}(E_n)
\leq \frac{1}{\varepsilon}\int_{\mathbb{R}^d} |f-f_n|\,d_{\mathrm{M}}
 \to 0 \ (n \to \infty), 
\end{split}
\end{equation*}
so $f_n \to f$ in measure. 
\end{proof}

\begin{lem}
If $\{f_n\} \to f$ {\rm in measure}, then there exists a subsequence 
$\{f_{n(k)}\}$  
that converges to $f$ almost everywhere.
\end{lem}

\begin{proof}
From $f_n \to f$ in measure, for each $k \in \mathbb{N}$, 
there is 
$n(k) \in \mathbb{N}$ such that 
\[
\overline{\mu}(\{x \in {\mathbb{R}^d} \ : \ |f(x)-f_n(x)| >2^{-k}\}) < 2^{-k} \ \ 
(\forall n \geq n(k)). 
\]
Choose $n(1) < n(2) <\cdots$ and set  
\[
\mathcal{N}_k:=\{
x \in {\mathbb{R}^d} \ : \ |f(x) -f_{n(k)}(x)| > 2^{-k}
\}, \ \  
\mathcal{N}:=\limsup_{k \to \infty}\mathcal{N}_k
=\bigcap_{\ell=1}^{\infty} \bigcup_{k=\ell}^{\infty}\mathcal{N}_k
\]
Then 
\begin{equation*} 
\overline{\mu}(\mathcal{N}) 
\leq 
\overline{\mu}\Bigl (\,\bigcup_{k={\ell}}^{\infty}\mathcal{N}_k \,\Bigr )
\leq \sum_{k=\ell}^{\infty}\overline{\mu}(\mathcal{N}_k)
= \frac{1}{2^{\ell-1}} \to 0\quad  (\ell \to \infty), 
\end{equation*}
so $\overline{\mu}(\mathcal{N})=0$. 
If $x \not \in \mathcal{N}$, then 
$x \in \bigcup\limits_{\ell=1}^{\infty} \bigcap\limits_{k=\ell}^{\infty}\mathcal{N}_k^c$; 
hence for $k \geq \ell$, 
$|f(x) - f_{n(k)}(x)| \leq 2^{-k}$, and thus 
$f_{n(k)}(x) \to f(x)\  (k \to \infty)$. 
\end{proof}

\section{Nonnegative Measurable Functions}

Since the Mimura integral was introduced as an extension of the Riemann integral, 
we would like it to coincide with the \emph{measure of the subgraph}
\begin{equation} \label{subgraph}
G(f):=\{(x,t) \in \mathbb{R}^{d+1}: x \in \mathbb{R}^d,\ 0 < t < f(x)\}. 
\end{equation}
Accordingly, we focus on nonnegative functions $f\geq 0$ 
for which the measure of this set is defined.

\begin{defn}[Subgraph set of $f$]
For $f \geq 0$, the set defined in \eqref{subgraph} is called 
the \emph{subgraph} of $f$. 
\end{defn}

\begin{prop}\label{measure of subgraph}
We have 
\[ 
G(f) \in \mathcal{M}(\mathbb{R}^{d+1}) \iff f \in \mathcal{M}^1(\mathbb{R}^d)
\] 
and in this case 
\begin{equation} \label{MSG}
\mu(G(f))=\int_{\mathbb{R}^d}f(x)\,d_{\mathrm{M}}x. 
\end{equation}
\end{prop}

\begin{rem}
By Proposition \ref{measure of subgraph}, the integral is determined by the function $f$ and the measure $\mu$. Hence it is also appropriate to write
\[
\int_{\mathbb{R}^d} f(x)\,d_{\mathrm{M}}x
=\int_{\mathbb{R}^d} f(x)\,d\mu(x).
\]
This notation is standard for the Lebesgue integral, and in \S 13 we will show that the Mimura integral is equivalent to the Lebesgue integral.
\end{rem}

\begin{proof} {(i) $f \in \mathcal{M}^1(\mathbb{R}^d) 
\implies G(f) \in \mathcal{M}(\mathbb{R}^{d+1})$}\,:

\vspace{1ex}

There exist increasing step functions 
$
\varphi_n \uparrow f\text{ a.e. on }\mathbb{R}^d$ with 
$\varphi_n \to f\text{ in }\mathcal{M}^1(\mathbb{R}^d)$. 
Fix $n$. Write 
\[ \varphi_n:=\sum\limits_{m=1}^N\alpha_m \chi_{I_m},\ I_m \in \mathscr{I}(\mathbb{R}^d) \]
Then, 
\[
G(\varphi_n)=\bigsqcup_{m=1}^N\{(x,t) \in \mathbb{R}^{d+1}: x \in I_m, 0<t<\alpha_m\}. 
\]
a disjoint union of $(d+1)$-dimensional intervals, so 
\[
\mu(G(\varphi_n))=\sum_{m=1}^N\alpha_m|I_m|
=\int_{\mathbb{R}^d}\varphi_n(x)\,d_{\mathrm{M}}x. 
\]
By Proposition \ref{continuity}
and the fact that
$G(\varphi_n)\uparrow G(f)$ a.e., 
the left-hand side converges to $\mu(G(f))$, 
while $\varphi_n \to f\text{ in }\mathcal{M}^1(\mathbb{R}^d)$ 
yields convergence of the right-hand side to 
$\displaystyle \int_{\mathbb{R}^d}f\,d_{\mathrm{M}}$. 
Letting $n \to \infty$ gives \eqref{MSG}. 

\smallskip

\noindent
{(ii) $G(f) \in \mathcal{M}(\mathbb{R}^{d+1}) \implies f \in \mathcal{M}^1(\mathbb{R}^d)$}\,: 

\vspace{1ex}
Since finite-measure sets can be approximated by disjoint countable unions of rectangles from inside and outside, for any $\varepsilon>0$
there exist
\[ 
C:=\bigsqcup_{m=1}^{\infty}I_m \times (0,\gamma_m), 
\quad 
D:=\bigsqcup_{m=1}^{\infty}I_m \times (0,\delta_m),
\] 
with $\{I_m\} \subset \Pi(\mathbb{R}^d)$, such that 
\[
C \subset G(f) \subset D, \quad  0\leq \mu(D)-\mu(C)<\varepsilon. 
\]
Let $\alpha_m:=\essinf\limits_{x \in I_m}f(x)$ and 
$\beta_m:=\esssup\limits_{x \in I_m}f(x)$, and define  
\[
A:=\bigsqcup_{m=1}^{\infty}I_m \times (0,\alpha_m),
\quad 
B:=\bigsqcup_{m=1}^{\infty}I_m \times (0,\beta_m). 
\]
Then 
$C \subset A \subset G(f) \subset B \subset D$. 
Indeed, if for some $m$ we had $\gamma_m>\alpha_m$, then by the definition of the essential infimum,
for any $t\in(\alpha_m,\gamma_m)$ the set $\{x\in I_m:\ f(x)\le t\}$ has positive measure.
Hence $\{(x,t):x\in I_m,\ 0<t<\gamma_m\}$ contains points with $t\ge f(x)$, so it cannot be
contained in $G(f)$, contradicting $C\subset G(f)$. Thus $\alpha_m\le\gamma_m$ for all $m$,
and similarly $B\subset D$.
By monotonicity, 
\[
\mu(C) \leq \mu(A) \leq \mu(G(f)) \leq \mu(B) \leq \mu(D). 
\]
Hence 
\begin{equation*} 
\begin{split} 
0 \leq 
\mu(B)-\mu(A) 
 = U(f, \{I_m\})-L(f, \{I_m\})
 \leq \mu(D)-\mu(C)<\varepsilon
\end{split}
\end{equation*}
and by Theorem \ref{Riemann conditions} we conclude 
$f \in \mathcal{M}^1(\mathbb{R}^d)$. 
\end{proof}

\begin{defn}[Measurable functions]
We call $f$ \emph{measurable} if $G(f) \in \mathcal{M}^*(\mathbb{R}^{d+1})$, 
and write $f \in \mathcal{M}^0(\mathbb{R}^d)$. 
\end{defn}

The next result shows that 
$\mathcal{M}^0(\mathbb{R}^d)$ 
is closed under limit operations.

\begin{prop}
Let 
$f, g \in \mathcal{M}^0(\mathbb{R}^d)$ with $f, g \geq0$, 
$\{f_n\}\subset \mathcal{M}^0(\mathbb{R}^d)$ with $f_n\geq 0$. 
Then the following functions all belong to 
$\mathcal{M}^0(\mathbb{R}^d)$. 

\begin{itemize}

\item[\rm{(i)}] 
$\displaystyle 
\overline{f}(x):=\sup_{n \in \mathbb{N}}f_n(x),\ 
\underline{f}(x):= \inf_{n \in \mathbb{N}} f_n(x) 
$

\item[\rm{(ii)}] 
$\displaystyle 
\overline{g}(x):=\limsup_{n \to \infty} f_n(x), \ 
\underline{g}(x):=\liminf_{n \to \infty} f_n(x)$

\end{itemize}

\end{prop}

\begin{proof}
(i) We have 
\begin{equation*} 
\begin{split} 
G(\overline{f})&= \{(x,t) : x \in \mathbb{R}^d, 0 < t < \sup_{n \in \mathbb{N}}f_n(x)\}\\[1ex]
&=\bigcup_{n=1}^{\infty}\{(x,t) : x \in \mathbb{R}^d, 0 < t < f_n(x)\}
= \bigcup_{n=1}^{\infty}G(f_n). \\[1ex]
G(\underline{f})&= \{(x,t) : x \in \mathbb{R}^d, 0 < t < \inf_{n \in \mathbb{N}}f_n(x)\}\\[1ex]
&=\bigcap_{n=1}^{\infty}\{(x,t) : x \in \mathbb{R}^d, 0 < t < f_n(x)\}
= \bigcap_{n=1}^{\infty}G(f_n). 
\end{split}
\end{equation*}
Since each $G(f_n) \in \mathcal{M}^*(\mathbb{R}^{d+1})$ and 
$\mathcal{M}^*(\mathbb{R}^{d+1})$
is a $\sigma$-algebra (Proposition \ref{sigma-algebra}), 
both $G(\overline{f})$ and $G(\underline{f})$ 
belong to $\mathcal{M}^*(\mathbb{R}^{d+1})$, 
hence $\overline{f}, \underline{f} \in \mathcal{M}^0(\mathbb{R}^d)$. \\[1ex]
\noindent 
(ii) Since 
\[
\overline{g}(x)
= \inf_{n \in \mathbb{N}}\sup_{k \geq n}f_k(x), 
\quad 
\underline{g}(x)
= \sup_{n \in \mathbb{N}}\inf_{k \geq n}f_k(x)
\]
the claim follows from (i). 
\end{proof}

\begin{prop}[Approximation by $\mathcal{M}^1(\mathbb{R}^d)$]
\label{approximation by integrable functions}
Let 
$f \in \mathcal{M}^0(\mathbb{R}^d)$ with $f \ge 0$. 
Then there exists $\{f_n\}\subset\mathcal{M}^1(\mathbb{R}^d)$
such that 
$f_n(x)\to f(x)$ a.e.\,$\mathbb{R}^d$. 
\end{prop}

\begin{proof}
Set 
$f_n(x):=\min\{f(x),n\}\cdot \chi_{\{|x| < n\}}(x)$. Then
\[
G(f_n)=G(f)\cap\bigl(\{|x| < n\}\times(0,n)\bigr)\in \mathcal{M}^*(\mathbb{R}^{d+1}), 
\]
and 
\[
G(f_n)=G(f)\cap\bigl(\{|x|<n\}\times(0,n)\bigr)\subset (-n,n)^d\times(0,n),
\]
hence
\[
\mu\bigl(G(f_n)\bigr)\le \mu\bigl((-n,n)^d\times(0,n)\bigr)
=2^d n^{d+1}<\infty.  
\]
so $G(f_n)\in\mathcal{M}(\mathbb{R}^{d+1})$.  
By Proposition \ref{measure of subgraph}, 
$f_n\in\mathcal{M}^1(\mathbb{R}^d)$. 
Moreover $G(f_n)\uparrow G(f)$ and $f_n\uparrow f$ a.e., 
hence $\mu(G(f_n))\uparrow\mu(G(f))$. 
\end{proof}

Finally, the following definition assigns an integral to every 
$f \in \mathcal{M}^0(\mathbb{R}^d)$ 
and makes the integral stable under limits. 

\begin{defn}[Integral of a measurable function]
If $f \in \mathcal{M}^0(\mathbb{R}^d)$ but 
$f \notin \mathcal{M}^1(\mathbb{R}^d)$, define  
\[
\int_{\mathbb{R}^d}f(x)\,d_{\mathrm{M}}x:=\infty. 
\]
\end{defn}

\section{Convergence Theorems for Nonnegative Functions}

Since 
$\mathcal{M}^0(\mathbb{R}^d)$ 
is closed under limit operations, 
we obtain the usual convergence theorems and corollaries. 
The following monotone convergence theorem is an immediate 
consequence of the $\sigma$-additivity of Lebesgue measure.

\begin{thm}[Monotone convergence theorem]

\label{monotone}

Assume: 

\begin{itemize}

\item[\rm{(i)}] $\{f_n\} \subset \mathcal{M}^0(\mathbb{R}^d)$. 

\item[\rm{(ii)}] $0 \leq f_n(x) \leq f_{n+1}(x)$ a.e.\ 
for all $n \in \mathbb{N}$.

\item[\rm{(iii)}] $f_n (x) \to f(x)$ a.e.\ as $n \to \infty$. 

\end{itemize}
Then, $f \in \mathcal{M}^0(\mathbb{R}^d)$ and 
\[
\lim_{n \to \infty}\int_{\mathbb{R}^d} f_n(x)\,d_{\mathrm{M}}x
=\int_{\mathbb{R}^d} f(x)\,d_{\mathrm{M}}x. 
\]
\end{thm}

\begin{proof} 
Since 
$
G(f_n)\uparrow G(f)
$, 
it follows from Proposition \ref{continuity} that 
$\mu(G(f))=\lim\limits_{n \to \infty}\mu(G(f_n))$, 
which yields the claim via Proposition \ref{measure of subgraph}. 
\end{proof}

\begin{thm}[Fatou's lemma]
\label{Fatou}

Assume:

\noindent
\begin{itemize}

\item[\rm{(i)}] $\{f_n\} \subset \mathcal{M}^0(\mathbb{R}^d)$. 

\item[\rm{(ii)}] $0 \leq f_n(x)$ a.e. 
for all $n \in \mathbb{N}$.

\end{itemize}
Then, 
\[
\int_{\mathbb{R}^d}\liminf_{n \to \infty}f_n(x)\,d_{\mathrm{M}}x 
\leq \liminf_{n \to \infty}\int_{\mathbb{R}^d} f_n(x)\,d_{\mathrm{M}}x. 
\]
\end{thm}

\begin{proof}
Apply Theorem \ref{monotone} to
$g_n(x):=\inf\limits_{k \geq n}f_k(x)$. 
\end{proof}

\begin{thm}[Dominated convergence for nonnegative functions]
\label{DCT4NN}
Assume: 

\noindent
\begin{itemize}

\item[\rm{(i)}] $\{f_n\}\subset \mathcal{M}^0(\mathbb{R}^d)$ and $f_n(x)\to f(x)$ a.e.

\item[\rm{(ii)}] $0\leq f_n(x)\leq g(x)$ a.e.\ for all $n \in \mathbb{N}$. 

\item[\rm{(iii)}] $g \in \mathcal{M}^1(\mathbb{R}^d)$. 

\end{itemize}
Then, $\{f_n\}\subset \mathcal{M}^1(\mathbb{R}^d)$, 
$f \in \mathcal{M}^1(\mathbb{R}^d)$ and 
\[
\lim_{n \to \infty}\int_{\mathbb{R}^d} f_n(x)\,d_{\mathrm{M}}x
=\int_{\mathbb{R}^d}f(x)\,d_{\mathrm{M}}x. 
\]

\end{thm}

\begin{proof}
Taking Fatou's lemma into account, it suffices to show 
\[
\limsup_{n \to \infty}\int_{\mathbb{R}^d} f_n(x)\,d_{\mathrm{M}}x
\leq \int_{\mathbb{R}^d}f(x)\,d_{\mathrm{M}}x. 
\]
This follows by applying Fatou's lemma to the nonnegative functions 
$g(x)-f_n(x) \geq 0$.  
\end{proof}

\begin{prop}[Layer-cake formula]
If $f \in \mathcal{M}^1(\mathbb{R}^d)$ with 
$f \geq 0$, then  
\[
\int_{\mathbb{R}^d}f(x)\,d_{\mathrm{M}}x
=\int_0^{\infty}
\mu(\{x: f(x)>t\})\,d_{\mathrm{M}}t. 
\]
\end{prop}

\begin{proof}
Choose step functions 
$\varphi_n \uparrow f$ a.e. and 
$\varphi_n \to f$ in $\mathcal{M}^1(\mathbb{R}^d)$ 
(Proposition \ref{approximation by step functions}). 
Write
\[
\varphi_n(x) = \sum_{m=1}^{N(n)}\alpha_{m,n}\chi_{I_m^n}, 
\quad 
\{\varphi_n > t\}= \bigsqcup_{\{m: \alpha_{m,n}>t\}}I_m^n
\]
so 
\[
\mu(\{\varphi_n >t \})=\sum_{m=1}^{N(n)}\chi_{[0,\alpha_{m,n})}(t)|I_m^n|. 
\]
Thus 
\[
\int_0^{\infty}\mu(\{\varphi_n >t\})\,d_{\mathrm{M}}t
= \sum_{m=1}^{N(n)}|I_m^n|
\int_0^{\infty}\chi_{[0,\alpha_{m,n})}(t)\,d_{\mathrm{M}}t
= 
\sum_{m=1}^{N(n)}\alpha_{m,n}|I_m^n|=\int_{\mathbb{R}^d}\varphi_n\,d_{\mathrm{M}}. 
\]
$\{\varphi_n >t\} \uparrow \{f>t\}=\bigcup\limits_{n=1}^{\infty}\{\varphi_n>t\}$ 
for each $t \geq 0$, 
Proposition \ref{continuity} implies 
$\mu(\{\varphi_n >t\}) \uparrow \mu(\{f>t\})$. 
Theorem \ref{monotone} gives the result. 
\end{proof}

\begin{thm}[Tonelli's theorem]
Let 
$f \in \mathcal{M}^0(\mathbb{R}^{d_1} \times \mathbb{R}^{d_2})$. 
Then 
\begin{equation*} 
\begin{split} 
& \int_{\mathbb{R}^{d_1}\times \mathbb{R}^{d_2}}f(x,y)\,d_{\mathrm{M}}(x,y) \\
&=\int_{\mathbb{R}^{d_1}}
\left ( 
\int_{\mathbb{R}^{d_2}}f(x,y)\,d_{\mathrm{M}}y
\right )
\,d_{\mathrm{M}}x
=\int_{\mathbb{R}^{d_2}}
\left ( 
\int_{\mathbb{R}^{d_1}}f(x,y)\,d_{\mathrm{M}}x
\right )
\,d_{\mathrm{M}}y. 
\end{split}
\end{equation*}
\end{thm}

\begin{proof}
By Propositions \ref{approximation by step functions} and \ref{approximation by integrable functions}, there exist step functions $\varphi_n$ with 
$G(\varphi_n) \uparrow G(f)$. 
Write 
\[
\varphi_n(x) = \sum_{m=1}^{N(n)}\alpha_{m,n}\chi_{I_m^n \times J_m^n}(x,y)
= \sum_{m=1}^{N(n)}\alpha_{m,n}
\chi_{I_m^n}(x)
\chi_{J_m^n}(y)
\]
For such $\varphi_n$, 
\begin{equation*} 
\begin{split} 
\int_{\mathbb{R}^{d_1}}
\left (
\int_{\mathbb{R}^{d_2}} \varphi_n(x,y) \,d_{\mathrm{M}}y
\right )\,d_{\mathrm{M}}x
& = 
\int_{\mathbb{R}^{d_1}}
\left ( 
\sum_{m=1}^{N(n)}\alpha_{m,n}
\chi_{I_m^n}(x)
|J_m^n| \right )
\,d_{\mathrm{M}}x \\
& =
\sum_{m=1}^{N(n)}\alpha_{m,n}
|J_m^n|
\int_{\mathbb{R}^{d_1}}
\chi_{I_m^n}(x)
\,d_{\mathrm{M}}x \\
& = 
\sum_{m=1}^{N(n)}\alpha_{m,n}
|I_m^n|
|J_m^n|
= 
\int_{\mathbb{R}^{d_1}\times \mathbb{R}^{d_2}}\varphi_n(x,y)\,d_{\mathrm{M}}(x,y). 
\end{split}
\end{equation*}
Since $\varphi_n$ increases pointwise to $f$,  
monotone convergence applies to both iterated integrals and to the integral over the product space, yielding the stated identities for $f$. 
\end{proof}

\begin{rem}[Streamlined measurability]
In the Mimura integral framework, measurability is defined via the subgraph $G(f)$. 
Hence for nonnegative functions the proof of Tonelli's theorem requires no separate measurability lemmas:
a step function approximation $\varphi_n\uparrow f$ 
and $\sigma$-additivity {\rm(}continuity from below{\rm)} already yield 
$\mu(G(f))=\lim_n\mu(G(\varphi_n))$, from which Tonelli's Theorem follows directly.
\end{rem}

\section{Integrals of Real and Complex-Valued functions}

Using linearity and reduction to nonnegative functions, 
we extend the integral to real and complex valued functions.

\begin{defn}[Integral of a real valued function]
Let $f:\mathbb{R}^d \to \mathbb{R}$. 
We declare $f \in \mathcal{M}^1(\mathbb{R}^d)$ iff 
\[
f^+ :=\frac{|f|+f}{2} \in \mathcal{M}^1(\mathbb{R}^d) 
\quad 
\text{ and }
\quad 
f^{-}:=\frac{|f|-f}{2} \in \mathcal{M}^1(\mathbb{R}^d)
\qquad (f=f^+-f^-). 
\]
For such $f$, define
\[
\int_{\mathbb{R}^d}f(x)\,d_{\mathrm{M}}x
:=\int_{\mathbb{R}^d}f^+(x)\,d_{\mathrm{M}}x-\int_{\mathbb{R}^d}f^-(x)\,d_{\mathrm{M}}x. 
\] 
\end{defn}

\begin{rem}
$
f^+(x)=\max{\{0,f(x)\}}$, $f^-(x)=\max{\{0,-f(x)\}}$. 
\end{rem}

\begin{defn}[Integral of a complex-valued function]
Let $f:\mathbb{R}^d \to \mathbb{C}$, write 
\[
f(x)=u(x) + iv(x), \quad 
u, v:\mathbb{R}^d \to \mathbb{R}. 
\]
We set 
$f \in \mathcal{M}^1(\mathbb{R}^d)$ iff 
$u, v \in \mathcal{M}^1(\mathbb{R}^d)$, and define
\[
\int_{\mathbb{R}^d}f(x)\,d_{\mathrm{M}}x:=\int_{\mathbb{R}^d}u(x)\,d_{\mathrm{M}}x 
+i\int_{\mathbb{R}^d}v(x)\,d_{\mathrm{M}}x. 
\]
\end{defn}

\begin{prop}[Basic properties of the integral]
Let $E \in \mathcal{M}^*(\mathbb{R}^d)$, 
$f, g \in \mathcal{M}^1(E)$, $\alpha, \beta \in \mathbb{C}$ and 
$\{E_n\}\subset \mathcal{M}^*(\mathbb{R}^d)$ 
with $E_m\cap E_n=\varnothing$ for $m \neq n$. 
Then:  

\begin{itemize}

\item[\rm{(i)}] $\displaystyle \int_E \alpha f(x) + \beta g(x)\,d_{\mathrm{M}}x 
= \alpha \int_E f(x)\,d_{\mathrm{M}}x + \beta \int_E g(x)\,d_{\mathrm{M}}x$.

\item[\rm{(ii)}] 
$\displaystyle \left | \int_{E}f(x)\,d_{\mathrm{M}}x\right |
\leq 
\int_{E}|f(x)|\,d_{\mathrm{M}}x
$.

\item[\rm{(iii)}] 
$\displaystyle \int_{\bigsqcup\limits_{n=1}^{\infty}E_n}f(x)\,d_{\mathrm{M}}x
= \sum_{n=1}^{\infty}\int_{E_n} f(x)\,d_{\mathrm{M}}x, 
\quad \forall\,f \in \mathcal{M}^1(\mathbb{R}^d)
$.

\item[\rm{(iv)}] If $f$ and $g$ are real-valued on $E$ and 
$f\leq g$ a.e. on $E$, then 
$\displaystyle \int_E f(x)\,d_{\mathrm{M}}x \leq \int_E g(x)\,d_{\mathrm{M}}x$. 

\end{itemize}

\end{prop}

\begin{thm}[Dominated convergence theorem]
\label{DCT} 
Assume: 

\noindent
\begin{itemize}

\item[\rm{(i)}] $\{f_n\}\subset \mathcal{M}^0(\mathbb{R}^d)$ and $f_n(x)\to f(x)$ a.e.

\item[\rm{(ii)}] $|f_n(x)|\leq g(x)$ a.e.\ for all $n \in \mathbb{N}$. 

\item[\rm{(iii)}] $g \in \mathcal{M}^1(\mathbb{R}^d)$. 

\end{itemize}
Then, $\{f_n\}\subset \mathcal{M}^1(\mathbb{R}^d)$, 
$f \in \mathcal{M}^1(\mathbb{R}^d)$ and 
\[
\lim_{n \to \infty}\int_{\mathbb{R}^d} f_n(x)\,d_{\mathrm{M}}x
=\int_{\mathbb{R}^d}f(x)\,d_{\mathrm{M}}x. 
\]

\end{thm}

\begin{proof}
Since $|f_n|= f_n^++f_n^- \leq g$, we have $0 \leq f_n^+, f_n^- \leq g$. 
Apply Theorem \ref{DCT4NN} to the sequences $\{f_n^+\}$ and $\{f_n^-\}$: 
\begin{equation*} 
\begin{split} 
\lim_{n \to \infty}\int_{\mathbb{R}^d} f_n\,d_{\mathrm{M}}
& = \lim_{n \to \infty}\int_{\mathbb{R}^d} f_n^{+}-f_n^-\,d_{\mathrm{M}}
= \lim_{n \to \infty}\int_{\mathbb{R}^d} f_n^{+}\,d_{\mathrm{M}}- 
\lim_{n \to \infty}\int_{\mathbb{R}^d} f_n^-\,d_{\mathrm{M}} \\
&= \int_{\mathbb{R}^d} f^+\,d_{\mathrm{M}} 
- \int_{\mathbb{R}^d} f^-\,d_{\mathrm{M}}
= \int_{\mathbb{R}^d} f\,d_{\mathrm{M}}. 
\end{split}
\end{equation*}
\end{proof}

\begin{thm}[Fubini's theorem]
If $f \in \mathcal{M}^1(\mathbb{R}^{d_1} \times \mathbb{R}^{d_2})$, then 
\begin{equation*} 
\begin{split} 
& \int_{\mathbb{R}^{d_1}\times \mathbb{R}^{d_2}}f(x,y)\,d_{\mathrm{M}}(x,y) \\
&=\int_{\mathbb{R}^{d_1}}
\left ( 
\int_{\mathbb{R}^{d_2}}f(x,y)\,d_{\mathrm{M}}y
\right )
\,d_{\mathrm{M}}x
=\int_{\mathbb{R}^{d_2}}
\left ( 
\int_{\mathbb{R}^{d_1}}f(x,y)\,d_{\mathrm{M}}x
\right )
\,d_{\mathrm{M}}y
\end{split}
\end{equation*}
\end{thm}

\begin{proof}
Apply Tonelli's theorem to $f^+$ and $f^-$. 
\end{proof}

\section{Relation of Riemann and Mimura integrals}

\begin{prop}
Every Riemann integrable function is Mimura integrable, and the two integrals coincide.
\end{prop}

\begin{proof}
It suffices to treat a nonnegative $f$ on 
$I \in \mathscr{I}(\mathbb{R}^d)$. 
Denote the Riemann integral by 
$\displaystyle 
\int_{I}f(x)\,d_{\mathrm{R}}x. 
$
For any $A \subset \mathbb{R}^d$, 
\[
\inf_{x \in A} f(x) \leq \essinf_{x \in A} f(x) \leq \esssup_{x \in A} f(x) \leq  \sup_{x \in A} f(x)
\]
Since Mimura's partitions are at least as flexible, we obtain 
\[
\underline{\int_I}f(x)\,d_{\mathrm{R}}x
\leq 
\underline{\int_I}f(x)\,d_{\mathrm{M}}x
\leq 
\overline{\int_I}f(x)\,d_{\mathrm{M}}x
\leq 
\overline{\int_I}f(x)\,d_{\mathrm{R}}x
\]
Thus Riemann integrability implies Mimura integrability with equal values. 
\end{proof}

\section{Equivalence with the Lebesgue integral}

In this section, we prove that the Mimura integral is equivalent to 
the Lebesgue integral. 
It suffices to treat nonnegative measurable functions. 

\subsection{Identification of the Measure Space
$(\mathbb{R}^d, \mathcal{M}^*(\mathbb{R}^d), \mu)$}
By Proposition \ref{outer measure1} and Corollary \ref{outer measure2}, 
our outer measure coincides with the Lebesgue outer measure, and 
$\mathcal{M}^*(\mathbb{R}^d)$ constructed via Carath\'eodory's criterion 
is precisely the collection of Lebesgue measurable sets. 
Hence our measure space 
$(\mathbb{R}^d, \mathcal{M}^*(\mathbb{R}^d), \mu)$ is 
the Lebesgue measure space.

\subsection{Equality of the Integrals}

Write the Lebesgue integral as 
$\displaystyle 
\int_{\mathbb{R}^d}f(x)\,d_{\mathrm{L}}x
$. 
Set
\[I_n:=(-n, n)^d, \quad 
f_n(x):=f(x) \chi_{I_n}(x) \quad 
a_k := \frac{k}{2^{n}} \quad (k = 0, 1, \ldots, n2^n), 
\]
and define $\varphi_n$ by 
\begin{equation} \label{approximation by measurable simple functions}
\varphi_n(x)
\displaystyle 
:=\sum_{k=1}^{n2^n} 
a_{k-1}\cdot  \chi_{\{ a_{k-1} \leq f_n < a_k\}}(x) + n \chi_{\{f_n \geq n\}}(x).
\end{equation}
Then $0 \leq \varphi_n(x) \uparrow f(x)$ as $n \to \infty$, 
and the Lebesgue and Mimura integrals of $\varphi_n$
coincide: 
\[
\int_{\mathbb{R}^d}\varphi_n\,d_{\mathrm{L}}
= \sum_{k=1}^{n2^n}a_{k-1}\mu(\{a_{k-1}\leq f_n <a_k\})
+ n \mu(\{f_n \geq n\})
= \int_{\mathbb{R}^d}\varphi_n\,d_{\mathrm{M}}. 
\]
By the definition of the Lebesgue integral and 
the monotone convergence theorem for Mimura integral (Theorem \ref{monotone}), 
\[
\int_{\mathbb{R}^d}f(x)\,d_{\mathrm{L}}x
:= \lim_{n \to \infty} \int_{\mathbb{R}^d}\varphi_n(x)\,d_{\mathrm{L}}x
=\lim_{n \to \infty}\int_{\mathbb{R}^d}\varphi_n(x)\,d_{\mathrm{M}}x
=\int_{\mathbb{R}^d}f(x)\,d_{\mathrm{M}}x. 
\]

\subsection{Equality of Measurable Functions: 
$\mathcal{M}^0(\mathbb{R}^d)=\mathcal{L}^0(\mathbb{R}^d)$}

\begin{defn}[Lebesgue measurable functions]
For a nonnegative $f$ on $\mathbb{R}^d$, 
we say $f$ is Lebesgue measurable if 
$\{f>a\} \in \mathcal{M}^*(\mathbb{R}^d)$ 
for every $a \geq 0$. 
Denote the class by $\mathcal{L}^0(\mathbb{R}^d)$. 
\end{defn}

\begin{prop}
$\mathcal{M}^0(\mathbb{R}^d)=\mathcal{L}^0(\mathbb{R}^d)$. 
\end{prop}

\begin{proof}
If $f \in \mathcal{M}^0(\mathbb{R}^d)$, i.e. 
$G(f) \in \mathcal{M}^*(\mathbb{R}^{d+1})$, 
then by Propositions \ref{approximation by step functions} and 
\ref{approximation by integrable functions} 
there exist step functions $\varphi_n$ with 
$G(\varphi_n) \uparrow G(f)$. 
For any $a \geq 0$ and $n \in \mathbb{N}$, 
$\{\varphi_n>a \} \in \mathcal{M}^*(\mathbb{R}^d)$ and 
since $\varphi_n \uparrow f$ a.e., 
\[ 
\{f>a\}= \bigcup_{n=1}^{\infty}\{\varphi_n>a\}\text{ a.e. } \]
Hence $f \in \mathcal{L}^0(\mathbb{R}^d)$.

Conversely, if $f \in \mathcal{L}^0(\mathbb{R}^d)$, 
then with $\varphi_n$ as in 
\eqref{approximation by measurable simple functions} we have 
$\varphi_n (x) \uparrow f(x)$ a.e. and 
\[
G(f)= \bigcup_{n=1}^{\infty}G(\varphi_n), 
\quad 
G(\varphi_n) 
\in \mathcal{M}(\mathbb{R}^{d+1})\subset \mathcal{M}^*(\mathbb{R}^{d+1})
\]
Therefore $f \in \mathcal{M}^0(\mathbb{R}^d)$. 

\end{proof}

\newpage 

\phantomsection

\end{document}